\documentclass[11pt]{amsart}

\usepackage{amsxtra,amssymb,amsthm,amsmath,amscd,url,listings,amsrefs}
\usepackage[utf8]{inputenc}
\usepackage{eucal}
\usepackage{fullpage}
\usepackage{scrtime}
\usepackage{hyperref}

\renewcommand{\leq}{\leqslant}
\renewcommand{\geq}{\geqslant}

\numberwithin{equation}{section}

\newcommand{\uple}[1]{\text{\boldmath${#1}$}}

\def\stacksum#1#2{{\stackrel{{\scriptstyle #1}}
{{\scriptstyle #2}}}}

\newcommand{\bfm}{\uple{m}}
\newcommand{\bfn}{\uple{n}}

\newcommand{\Cc}{\mathbf{C}}
\newcommand{\Nn}{\mathbf{N}}
\newcommand{\Aa}{\mathbf{A}}
\newcommand{\Bb}{\mathbf{B}}
\newcommand{\Dd}{\mathbf{D}}
\newcommand{\Zz}{\mathbf{Z}}
\newcommand{\Pp}{\mathbf{P}}
\newcommand{\Rr}{\mathbf{R}}
\newcommand{\Gg}{\mathbf{G}}
\newcommand{\Gm}{\mathbf{G}_{m}}

\newcommand{\Qq}{\mathbf{Q}}
\newcommand{\Fp}{{\mathbf{F}_p}}
\newcommand{\Fpt}{{\mathbf{F}^\times_p}}
\newcommand{\Ff}{\mathbf{F}}

\newcommand{\gnat}{g^\natural}

\newcommand{\HYPK}{\mathcal{K}\ell}

\newcommand{\mods}[1]{\,(\mathrm{mod}\,{#1})}

\newcommand{\wwd}{\mathcal{C}}
\newcommand{\haut}{\mathbf{G}}
\newcommand{\hautb}{\mathbf{B}}

\newcommand{\tsum}{\mathcal{S}}

\newcommand{\frtr}[3]{(\Tr{{#1}})({#2},{#3})}
\DeclareMathOperator{\hypk}{Kl}

\newcommand{\lmax}{\mathcal{M}}

\newcommand{\ra}{\rightarrow}
\newcommand{\lra}{\longrightarrow}

\newcommand{\injecte}{\hookrightarrow}

\newcommand{\fleche}[1]{\stackrel{#1}{\lra}}

\DeclareMathOperator{\rank}{rank}

\DeclareMathOperator{\frob}{\mathrm{Fr}}
\DeclareMathOperator{\tr}{\mathrm{tr}}
\DeclareMathOperator{\Gal}{Gal}

\DeclareMathOperator{\supp}{supp}

\DeclareMathOperator{\Tr}{tr}

\DeclareMathOperator{\swan}{Swan}

\DeclareMathOperator{\ft}{FT}
\DeclareMathOperator{\cond}{cond}
\DeclareMathOperator{\Ad}{Ad}
\DeclareMathOperator{\dual}{D}

\newcommand{\eps}{\varepsilon}
\renewcommand{\rho}{\varrho}

\DeclareMathOperator{\SL}{SL}
\DeclareMathOperator{\GL}{GL}
\DeclareMathOperator{\PGL}{PGL}
\DeclareMathOperator{\rmT}{T}
\DeclareMathOperator{\rmG}{G}
\DeclareMathOperator{\rmN}{N}
\DeclareMathOperator{\rmU}{U}

\DeclareMathOperator{\Sp}{Sp}

\newcommand{\sheaf}[1]{\mathcal{{#1}}}

\DeclareMathSymbol{\gena}{\mathord}{letters}{"3C}
\DeclareMathSymbol{\genb}{\mathord}{letters}{"3E}

\def\dblsum{\mathop{\sum \sum}\limits}

\def\multsum{\mathop{\sum\cdots \sum}\limits}

\def\intc{\frac{1}{2i\pi}\mathop{\int}\limits}

\theoremstyle{plain}
\newtheorem{theorem}{Theorem}[section]
\newtheorem*{theorem*}{Theorem}
\newtheorem{lemma}[theorem]{Lemma}
\newtheorem{corollary}[theorem]{Corollary}

\newtheorem{proposition}[theorem]{Proposition}

\theoremstyle{remark}

\newtheorem*{rem}{Remark}

\theoremstyle{definition}

\newtheorem{definition}[theorem]{Definition}

\newtheorem{remark}[theorem]{Remark}

\newcommand{\mcL}{\mathcal{L}}

\newcommand{\mcF}{\mathcal{F}}
\newcommand{\mcJ}{\mathcal{J}}
\newcommand{\mcI}{\mathcal{I}}

\renewcommand{\geq}{\geqslant}
\renewcommand{\leq}{\leqslant}
\renewcommand{\Re}{\mathfrak{Re}\,}
\renewcommand{\Im}{\mathfrak{Im}\,}

\newcommand{\refs}{\eqref}

\newcommand{\ov}[1]{\overline{#1}}

\newcommand\sumsum{\mathop{\sum\sum}\limits}

\begin{document}

\title{Algebraic trace functions over the primes}
 
\author{\'Etienne Fouvry}
\address{Universit\'e Paris Sud, Laboratoire de Math\'ematique\\
  Campus d'Orsay\\ 91405 Orsay Cedex\\France}
\email{etienne.fouvry@math.u-psud.fr}

\author{Emmanuel Kowalski}
\address{ETH Z\"urich -- D-MATH\\
  R\"amistrasse 101\\
  CH-8092 Z\"urich\\
  Switzerland} \email{kowalski@math.ethz.ch}

\author{Philippe Michel} \address{EPFL/SB/IMB/TAN, Station 8, CH-1015
  Lausanne, Switzerland } \email{philippe.michel@epfl.ch}

\date{\today,\ \thistime} 

\thanks{Ph. M. was partially supported by
  the SNF (grant 200021-137488) and the ERC (Advanced Research Grant
  228304). \'E. F. thanks ETH Z\"urich, EPF Lausanne and the Institut
  Universitaire de France for financial support.  }

\subjclass[2010]{11N05, 11N13, 11N32, 11N35, 11F11, 11T23, 11L05}

\keywords{Sums over primes, M\"obius function, Eisenstein series,
  trace functions of $\ell$-adic sheaves, Riemann Hypothesis over
  finite fields}

\begin{abstract}
  We study sums over primes of trace functions of $\ell$-adic
  sheaves. Using an extension of our earlier results on algebraic
  twist of modular forms to the case of Eisenstein series and bounds
  for Type II sums based on similar applications of the Riemann
  Hypothesis over finite fields, we prove general estimates with
  power-saving for such sums. We then derive various concrete
  applications. 
\end{abstract}

\maketitle
\setcounter{tocdepth}{1}
\tableofcontents

\section{Introduction}

Let $f(X)=P(X)/Q(X)$, where $P,Q\in\Zz[X]$, be a non-constant rational
function. If $p$ is a prime large enough so that $f(X)$ defines a
rational function on $\Fp$ by reduction modulo $p$, it follows from
the work of Weil that we have the estimate
$$
\sum_\stacksum{1\leq n\leq
  p}{(Q(n),p)=1}e\Bigl(\frac{P(n)\ov{Q(n)}}p\Bigr)\ll
p^{1/2}
$$ 
(where $Q(n)\ov{Q(n)}=1\mods p$) which exhibits considerable
cancellation in this exponential sum. It is a natural question, with
many potential applications, to ask whether such cancellation persists
when the sum is restricted to prime numbers $q$, either less than $p$
or over shorter intervals (longer intervals being usually easier to
handle).

In \cite[Théorème 1.1]{FMAnn}, Fouvry and Michel proved that this is
indeed almost always the case:

\begin{theorem}\label{primesumthmFM}
  Let $f=P/Q$, with $P$, $Q\in\Zz[X]$ coprime unitary polynomials. For
  every prime $p$ such that the reduction of $f$ modulo $p$ is not a
  polynomial of degree $\leq 1$, for every $X\leq p$ and every
  $\eta<1/32$, we have
$$
\sum_\stacksum{q\leq X,\ prime}{(Q(q),p)=1}
e\Bigl(\frac{P(q)\ov{Q(q)}}p\Bigr) \ll
X\Bigl(\frac{p}X\Bigr)^{7/32}p^{-\eta},
$$
where the implicit constant
depends only on $\eta$ and on the degrees of $P$ and $Q$.
\end{theorem}

Similar estimates were already known when $f$ is a polynomial (of
degree $>1$), but with the exponent $\eta$ depending on the degree of
$f$ and tending to $0$ as the latter increased (see,
e.g,~\cite{Hua,Harman}). Thus an important new feature in~\cite{FMAnn} was
to allow for the most general possible rational fractions, and for a
uniform $p$-power saving. One of the key input of the proof was an
essential use of Deligne's theory of higher dimensional algebraic
exponential sums.

There are, however, many other functions defined over $\Fp$ for which
one would like to have similar results. For instance, with
$f(X)=P(X)/Q(X)$ with $P$, $Q$ in $\Zz[X]$ as above, one may naturally
want to consider
$$
K(n)=
\begin{cases}
  \chi(f(n))&\text{if $(p,Q(n))=1$ },\\
  0&\hbox{if}\ p|Q(n),
\end{cases}
$$
for some non-trivial Dirichlet character $\chi$ modulo $p$ of order
$h\geq 2$, provided the rational function $f$ is not proportional to
an $h$-th power. Nevertheless, the only examples we are aware of
concerns the case where $f$ is a split polynomial of degree $\leq 2$
not vanishing at $0$ which was considered by
Karatsuba~\cite{Kar,Kar2,Kar3} (see also \cite{FGS}). In
Corollary~\ref{cor-mult-car}, we will prove a non-trivial bound for an
(almost) arbitrary rational function $f$.

Beyond additive and multiplicative characters, there are other
functions defined over finite fields which are now common tools in
number theory. A nice example is given by the (normalized)
hyper-Kloosterman sums in $m-1$ variables, introduced by Deligne and
studied by Katz in great detail in \cite{GKM}, which are defined by
$$
K(n)=\hypk_m(n;p)=\frac{1}{p^{\frac{m-1}2}}\multsum_\stacksum{x_1\cdots
  x_m=n}{x_i\in\Fp} e\Bigl(\frac{x_1+\cdots+x_m}p\Bigr),
$$
for some integer $m\geq 2$. This example was first considered by the third
author in~\cite{Michelthese,MichelInv,MichelDMJ}, who obtained a modest (yet
non-trivial) saving of $\frac{\log\log p}{\log p}$ over the trivial
bound $O(p/\log p)$ for the sum over primes $q<p$ of $\hypk_m(q;p)$.

\begin{remark}\label{introHNY}
  One can also wonder about the very recent generalizations of
  Kloosterman sums $\hypk^\rho_{\check\Gg}$ associated to the general
  Kloosterman sheaves defined, for an arbitrary split reductive group
  $\check\Gg$ and a representation $\rho$ of it, by Heinloth, Ng\^o and
  Yun~\cite{HNY}, where the case of the hyper-Kloosterman sums above
  corresponds to $\check\Gg=\GL_n$ with its standard
  representation. It is a sign of the generality of our results that
  they do apply very straightforwardly to this case, although the
  corresponding trace functions have not (yet) been made explicit for
  all\footnote{\ As Ng\^o kindly informed us, Yun has computed these
    sums explicitly for $\check\Gg=\mathrm{SO}(2n+1)$ and its standard
    representation: $\hypk_{\mathrm{SO}(3)}$ is the symmetric square
    of $\hypk_2$ and $\hypk_{\mathrm{SO}(2n+1)}$ for $n\geq 2$ is
    essentially the multiplicative convolution of
    $\hypk_{\mathrm{SO}(3)}$ and of two Kloosterman sums $\hypk_n$,
    see~\cite{Yun}.} $\check\Gg$!
\end{remark}

The common link between all these functions is that they are special
cases of the general class of functions we called {\em trace weights}
in \cite{FKM}, with bounded conductors. Precisely, we have the
following definition:

\begin{definition}[Trace weights]\label{def-admissible}
  For a prime $p$ and a prime $\ell\not=p$, an {\em isotypic trace
    sheaf modulo} $p$ is a geometrically isotypic $\ell$-adic Fourier
  sheaf $\mcF$ on $\Aa^1_{\Fp}$, in the sense
  of~\cite[Def. 8.2.2]{GKM}, which is pointwise pure of weight $0$.
\par
An {\em isotypic trace weight modulo} $p$ is the trace function
$$
K(x)=\iota(\frtr{\sheaf{F}}{\Fp}{x})
$$
for $x\in \Fp$ of an isotypic trace sheaf $\mcF$, this trace function
being seen as complex-valued by means of some fixed isomorphism
$\iota\,:\,\bar{\Qq}_{\ell}\ra \Cc$.
\end{definition}

To any middle-extension sheaf $\mcF$ on $\Aa^1_{\Fp}$ is associated
its \emph{analytic conductor}, a numerical invariant which measures
the complexity of $\mcF$. This is a positive integer defined by
$$
\cond(\mcF)=\rank(\mcF)+\sum_{x}(1+\swan_x(\mcF)),
$$
where $x$ ranges over the (finitely many) singularities of $\mcF$ in
$\Pp^1(\ov\Fp)$, i.e., those $x$ where $\mcF$ is not lisse, and
$\swan_x(\mcF)\geq 0$ is the Swan conductor of $\sheaf{F}$ at $x$
(see~\cite{GKM}). For an isotypic trace weight $K(n)$, we define the
conductor as the minimal conductor of an isotypic trace sheaf
$\sheaf{F}$ with trace function equal to $K(n)$ on $\Fp$.

\begin{remark} For example:
\begin{itemize}
\item[-] If $K(n)=e(P(n)/p)$ for a polynomial $P\in\Fp[X]$ of degree
  $<p$, the associated sheaf has conductor $\cond(\mcF)=\deg P+2$;
\item[-] If $K(n)=\chi(P(n))$ where $\chi$ is multiplicative and
  $P\in\Fp[X]$ a polynomial, then $\cond(\mcF)$ is bounded by $2$ plus
  the number of distinct zeros of $P$ in $\ov{\Fp}$;
\item[-] For the hyper-Kloosterman sums $K(n)=\hypk_m(n;p)$, the
  conductor is $m+3$;
\item[-] For the trace function of the $\ell$-adic Kloosterman sheaf
  associated to the adjoint representation of the split reductive
  group $\check\Gg$ in~\cite{HNY}, the conductor is bounded by
  $\dim(\check\Gg)+2+r(\check\Gg)$, where $r(\check\Gg)$ is the rank
  of $\check\Gg$ (by~\cite[p. 4, (3)]{HNY}: this sheaf is of dimension
  $\dim \Ad=\dim\check\Gg$, lisse on $\Gg_m$, tame at $0$ and with
  Swan conductor $r(\check\Gg)$ at $\infty$).
\end{itemize}
\end{remark}

Our main result in this paper is an estimate for any sum over primes
$q$ of an isotypic trace function modulo $p$, which is universal in
quality and gives power-saving whenever the length of the sum is
roughly comparable with $p$ on a logarithmic scale, in particular
allowing some sums over shorter intervals $q<p^{1-\theta}$ for some
$\theta>0$.  The only weights we cannot handle are those where the
corresponding estimate would be tantamount to a ``quasi-Riemann
Hypothesis'', i.e., a zero-free strip for some Dirichlet
$L$-functions.
\par
It will therefore be natural to say that $K(n)$ is an
\emph{exceptional weight modulo $p$} (for sums over primes) if it is
proportional to a weight
$$
K_{\chi,\psi}(n)=\chi(n)\psi(n)
$$
where $\chi$ (resp. $\psi$) is a multiplicative (resp. additive)
character modulo $p$, where either or both may be trivial.  
\par
Similarly, a sheaf $\mcF$ will be called \emph{exceptional} if it is
geometrically isotypic and geometrically isomorphic to a sum of copies
of a tensor product $\sheaf{L}_{\chi}\otimes\sheaf{L}_{\psi}$ of a
Kummer sheaf with an Artin-Schreier sheaf, so that its trace function
is exceptional. 
\par
We state the results both for standard and for smoothed sums over
primes. We will consider smooth test functions $V$, compactly
supported in $[1/2,2]$, such that
\begin{equation}\label{Vcond}
  x^jV^{(j)}(x)\ll Q^j
\end{equation}
for some $Q\geq 1$ and for any integer $j\geq 0$, where the implicit
constant depends on $j$.

\begin{theorem}[Trace weights vs. primes]\label{primesumthm}
  Let $K$ be an isotypic trace weight on $\Fp$ associated to some
  sheaf $\mcF$, and assume that $\mcF$ is not exceptional.  Let $V$ be
  a smooth function as above satisfying~\emph{(\ref{Vcond})} for some
  parameter $Q\geq 1$.
\par
For any $X\geq 2$, we have
\begin{align}\label{primesumsmooth}
\sum_{q\ \text{prime}}K(q)V\Bigl(\frac{q}X\Bigr)&\ll
QX(1+p/X)^{1/6}p^{-\eta},\\
\label{primesuminterval}
\sum_\stacksum{q\ \text{prime}}{q \leq X}K(q)&\ll
X(1+p/X)^{1/12}p^{-\eta/2},
\end{align}
for any $\eta<1/24$. The implicit constants depend only on $\eta$,
$\cond(\mcF)$ and the implicit constants
in~\emph{(\ref{Vcond})}. Moreover, the dependency on $\cond(\mcF)$ is
at most polynomial.
\end{theorem}

\begin{remark} 
  For $X=p$ one gets 
$$
\sum_\stacksum{q\ \text{prime}}{q <p}K(q)\ll
p^{1-1/48+\eps},
$$ 
and for general $X$ these bounds are non-trivial as long as the
conductor of $\mcF$ remains bounded and the range $X$ is greater that
$p^{3/4+\eps}$ for some $\eps>0$. Stronger results are available by
different methods for special $K$. For instance,
Bourgain~\cite{bourgainmore} and Bourgain-Garaev \cite{BG} have
obtained bounds for
$$
K(n)=e\Bigl(\frac{an+b{\ov{n}}^{k}}{p}\Bigr),\ k\in\Nn-\{0\},\
(b,p)=1,
$$
which are non-trivial as long as $X\geq p^{1/2+\eps}$ (see also
\cite{FoSh} for a survey of existing methods).
\end{remark}

Closely related to Theorem \ref{primesumthm} is the following
estimate:

\begin{theorem}[Trace weights vs. M\"obius]
  \label{moebiussumthm} Let $\mu$ denote the M\"obius function. With
  the same notations and hypotheses as in Theorem~\ref{primesumthm}, we
  have for $X\geq 2$
\begin{align*}
\sum_{n}\mu(n)K(n)V\Bigl(\frac{n}X\Bigr)&\ll 
QX(1+p/X)^{1/6}p^{-\eta},\\
\sum_{n \leq X}\mu(n)K(n)&\ll 
X(1+p/X)^{1/12}p^{-\eta/2},
\end{align*}
for any $\eta<1/24$, where the implicit constants depend only on
$\eta$, $\cond(\mcF)$ and the implicit constants
in~\emph{(\ref{Vcond})}, and the dependency on $\cond(\mcF)$ is at
most polynomial.
\end{theorem}

\begin{remark}
  The discrepancy in the power-saving exponents between the smoothed
  and unsmoothed sums in Theorem \ref{primesumthm} and
  \ref{moebiussumthm} is due to the growth of the Sobolev norms of the
  smooth functions approximating the characteristic function of the
  interval $[1,X]$, which is measured by the parameter $Q$ of the
  smoothed sums. This dependency shows up in the treatment of the sums
  of Type $I$ and $I_2$ below, through Theorem \ref{typeIsumthm} and
  the shape of the packet of Eisenstein series involved in the
  analysis. Any improvement in the exponent of the parameter $Q_W$ in
  these statements will yield an improvement for the unsmoothed sums,
  but it is also quite possible that other more advanced arguments
  could give such improvements. One should also remark however that
  the smallest range of $X$ for which the sums can be bounded
  non-trivially ($X\approx p^{3/4}$) is the same for the smoothed and
  unsmoothed sums (for a fixed smooth function $V$).
\end{remark}

\begin{remark} 
  Jean Bourgain pointed out that Theorem 2 of~\cite{BSZ} (along with the
  Note following it) can be used in conjonction with~(\ref{eq-correl})
  and Proposition~\ref{correlationprop} of the present paper to prove
  that if $K$ is an isotypic trace weight, we have, for any $\eps>0$ and for any $X\geq p^{1/2+\eps}$
$$
\sum_{n \leq X}\mu(n)K(n)=o(X),
$$
where the implicit constant depends on $\cond(\mcF)$ and $\eps$.  However, this
approach does not seem to yield a power saving, and does not seem to
apply if $\mu$ is replaced by $\Lambda$ or by the characteristic
function of the primes.
\end{remark}

\begin{remark}
  This theorem expresses an orthogonality property (i.e., the absence
  of correlation) between the M\"obius function and any isotypic trace
  weight modulo $p$ with bounded conductor. This fits with the
  philosophy of the ``M\"obius randomness principle'', formulated
  vaguely in~\cite[p. 338]{KI}, and with Sarnak's recent precise
  formulation in terms of orthogonality of the M\"obius function
  against function with low complexity, in the sense of entropy
  (see~\cite{sarnak}). Our result is in a slightly different context
  than Sarnak's conjecture, however, since the trace weights $K(n)$
  are defined modulo $p$, and an asymptotic statement follows only by
  taking, for each $p$, a different weight with some bound on the
  complexity, as measured by the conductor in our case.  Thus, our
  results are closer in spirit to those of Green~\cite{green} and
  Bourgain~\cite{bourgain}, which prove asymptotic orthogonality of
  the M\"obius function against, respectively, bounded depth boolean
  functions, and monotone boolean functions on binary hypercubes
  $\{0,1\}^N$ identified with $\{1,\ldots, 2^N\}$.
\par
In fact, it seems to be a very intriguing question (suggested by
Peter Sarnak) to understand which functions can arise as small linear
combinations of trace functions of low conductor (which, by linearity,
still satisfy Theorems~\ref{primesumthm}
and~\ref{moebiussumthm}). Another natural question is whether trace
functions have low complexity in a more algorithmic sense, and this
does not seem to be easy to answer. Only functions such as
$$
K(n)=e\Bigl(\frac{P(n)}p\Bigr) \text{ or }
K(n)=\Bigl(\frac{P(n)}p\Bigr)
$$
(for $P(X)\in\Zz[X]$ fixed and $(\frac\cdot p)$ the Legendre symbol)
seem to be obviously of low complexity for most meaning of the term,
and such functions as
$$
K(n)=\hypk_2(n;p)
$$
are far from being understood in this respect. One can certainly
expect the class of trace weights (and linear combinations with small
coefficients) to be very rich and fascinating (in this respect, there
are already hints in Deligne's conjecture about the number of trace
functions with various conditions, see
\cite{deligne-drinfeld,EK,FKM1.5} for this topic.)
\end{remark}

\subsection{First applications}

We present here some corollaries of Theorem~\ref{primesumthm} which
are obtained by applying it to specific weights $K$. This is only a
selection and we expect many more applications. We leave to the reader
to write down the corresponding statements involving the M\"obius
function, which follow from Theorem~\ref{moebiussumthm}.
\par
In the first result, we obtain a power saving in the sum of the error
term in the prime number theorem in arithmetic progressions over
residues classes modulo a prime which are values of some fixed
polynomial. Precisely, we define $E(X;p,a)$ by
$$
\pi(X;p,a) =\frac{\delta_p(a)}{\varphi(p)}\pi(X)+E(X;p,a),
$$
where $\delta_p(a)=0$ if $(a,p)\not=1$ and is $1$ otherwise.
\par
We will prove in \S \ref{sec:primepolynomials} :

\begin{corollary}\label{cor-poly-error-terms}
  Let $P\in \Zz[X]$ be a polynomial whose reduction modulo $p$ is
  squarefree and non-constant.
\par
\emph{(1)} We have
$$
\sum_{n\in\Fp}{E(X;p,P(n))}\ll X(1+p/X)^{1/12}p^{-\eta}
$$
for any $\eta<1/48$, where the implicit constant depends only on $\eta$
and $\deg P$.
\par
\emph{(2)} We have
$$
\sum_{a\in P(\Fp)}{E(X;p,a)}\ll X(1+p/X)^{1/12}p^{-\eta}
$$
for any $\eta<1/48$, where the implicit constant depends only on $\eta$
and $\deg P$.
\end{corollary}

Note that this corollary is trivial if $P$ is linear.  The restriction
to a squarefree polynomial could be relaxed, but some condition is
needed in the current state of knowledge since for $P=X^2$ we have the
interpretation
$$
\sum_{n\in\Fp}{E(X;p,n^2)}=
\sum_{q\leq X}\Bigl(\frac{q}{p}\Bigr)+O(1)
$$
in terms of average of the Legendre symbol over primes, from which we
cannot get power saving without using a quasi-Riemann Hypothesis for the
corresponding $L$-function.
\par
On the other hand, if we take $P=X^2-1$, the study of either of these
sums becomes equivalent to that of the sum
$$
\sum_{q<X}{\chi(q+1)}
$$
where $q$, again, runs over primes. This was estimated, as we
recalled, by Karatsuba~\cite{Kar}. 
\par
We can now generalize considerably this result of Karatsuba:

\begin{corollary}[Character sums over polynomially-shifted primes]
  \label{cor-mult-car}
  Let $f=P/Q$ be a rational function represented as a ratio of
  integral polynomials. Let $\chi$ be a non-trivial Dirichlet
  character of prime modulus $p$ and order $h\geq 2$. Assume that $f$
  modulo $p$ is not of the form
$$
cX^kg(X)^h
$$
for some $c\in\Fp^{\times}$, some $k\in\Zz$ and some
$g(X)\in\Fp(X)$. We then have
\begin{align*}
  \sum_{q\ \text{prime}}{\chi(f(q))V(q/X)}&\ll
  X(1+p/X)^{1/6}p^{-\eta}\\
  \sum_\stacksum{q\ \text{prime}}{q\leq X}{\chi(f(q))}&\ll
  X(1+p/X)^{1/12}p^{-\eta/2}
\end{align*}
for any $\eta<1/24$, where the implicit constant depends only on
$\eta$, $V$ and the degrees of $P$ and $Q$.
\end{corollary}

\begin{proof} 
  We will show that Theorem~\ref{primesumthm} is applicable. We first
  recall (see \cite[Chap. 4]{GKM}) that an Artin-Schreier sheaf
  $\mcL_\psi$ (for $\psi$ non trivial) is wildly ramified at $\infty$
  and unramified on $\Aa^1_{\Fp}$, while a Kummer sheaf $\mcL_{\chi}$
  (for $\chi$ non-trivial) is tamely ramified at $0$ and $\infty$ and
  unramified on ${\Gm}_{,\Fp}$. Both types of sheaves have rank $1$.
\par
The weight $K(n)=\chi(f(n))$ is the trace function of the tame, Kummer
sheaf $\sheaf{L}_{\chi(f)}$ which has (at most) $\deg P+\deg Q$
singularities, hence conductor bounded in terms of $\deg P$ and $\deg
Q$. Theorem~\ref{primesumthm} therefore applies when
$\sheaf{L}_{\chi(f)}$ is not exceptional. We now determine when this
is so.
\par
The sheaf $\sheaf{L}_{\chi(f)}$ is tamely ramified at
$x\in\Pp^1(\ov\Fp)$ if and only if $x$ is a zero or a pole of $f$ of
order not divisible by $h$. This implies that if $\sheaf{L}_{\chi(f)}$
is geometrically isomorphic to some $\mcL_\psi\otimes\mcL_{\chi'}$
then $\psi$ is trivial (otherwise $\sheaf{L}_{\chi(f)}$ would be
wildly ramified at $\infty$) and the zeros or poles of $f$ distinct
from $0$ and $\infty$ have order divisible by $h$; this means
precisely that $f$ is of the form $cX^kg(X)^h$.
\end{proof}

One can unify this corollary with Theorem~\ref{primesumthmFM} for
$$
K(n)=\chi(f(n))e\Bigl(\frac{g(n)}{p}\Bigr),
$$
getting cancellation for sums over primes provided $f$ (resp. $g$)
satisfies the assumptions of Corollary~\ref{cor-mult-car}
(resp. Theorem~\ref{primesumthmFM}). 

\subsection{Kloosterman sums at prime arguments}

Our last application involves the weights $K(n)$ which are related to
Kloosterman or hyper-Kloosterman sums $\hypk_m$. We first spell out
two very specific corollaries for the standard Kloosterman sum in one
variable:

\begin{corollary}  
  For every $0 < \eta < 1/48$ there exists $C(\eta)$ such that for
  every $p$, every $X\geq 2$ and every integer $n$ coprime with $p$,
  one has the inequalities
$$
\Bigl\vert \, \sum_{q<X,\, q\text{ prime} }\hypk_2(nq;p)\log q
\,\Bigr\vert \leq C(\eta) X(1+p/X)^{1/12}p^{-\eta}
$$
and
$$
\Bigl\vert \, \sum_{q<X, \, q\text{ prime
  }}\hypk_2(n^2q^2;p)e\Bigl(\frac{2nq}p\Bigr)\log q \,\Bigr\vert \leq
C(\eta) X(1+p/X)^{1/12}p^{-\eta}.
$$
\end{corollary}

These two bounds improve~\cite[Lemmas 6.1, 6.2, 6.3]{ILS} when $c=p$
is a prime and when $X$ is near and possibly a bit smaller than $p$
(these results were proved in~\cite{ILS} assuming the Generalized
Riemann Hypothesis for Dirichlet characters). Using the methods of
\cite{ILS}, one can use the second bound and the Petersson formula to
increase the size of the support of the Fourier transform $\widehat
\Phi$ of the test functions $\Phi$ in the problem of computing the
distribution of low-lying zeros (average $1$-level density) of the
symmetric square $L$-functions $L(\mathrm{sym}^2f,s)$ for $f$ in the
family of holomorphic newforms of prime level $p\ra +\infty$ and
weight $k$: with notation as in~\cite{ILS} (except that they denote
the level by $N$), \emph{there exists $\delta>0$ such that for any
  $\Phi\in\mathcal{S}(\Rr)$ with the support of $\widehat\Phi$ in
  $]-1/2-\delta,1/2+\delta[$, one has}
 $$
\lim_{p\ra\infty}\frac{1}{|H^{\star}_k(p)|}\sum_{f\in
  H^{\star}_k(p)}D(\mathrm{sym}^2f,\Phi)=\int_\Rr
\Phi(x)W(\mathrm{Sp})(x)dx.
$$
\par
The possibility of such an improvement was known to the authors of
\cite{ILS} (see \cite[Remark C, p. 61]{ILS}), though their method was
different.

The consideration of powers of hyper-Kloosterman sums allows us to
strengthen the results~\cite{MichelInv,MichelDMJ} concerning the
existence of hyper-Kloosterman sums with large absolute value modulo a
product of two primes:

\begin{corollary}\label{largesums} 
For any $m\geq 2$,
there exists a constant $\alpha_m>0$ such that
$$
\sum_{c\leq X}\Lambda_2(c)|\hypk_m(1;c)|\geq (\alpha_m+o_m(1))X\log X,
$$
were $\Lambda_2(c)=(\mu\star\log^2)(c)$ denotes the von Mangoldt
function of order $2$, which is supported over integers with at most
two prime factors.
\end{corollary}

This corollary shows that the normalized hyper-Kloosterman sums
$\hypk_m(1;c)$ whose modulus is a product of at most two primes have
their size $\gg_m 1$ for a positive proportion of such moduli (when
these are weighted by
$\Lambda_2$). In~\cite{Michelthese,MichelInv,MichelDMJ}, the
lower-bound was of order $X$,
and by adding the missing logarithmic factor, we obtain the right
order of magnitude. This answers a question of Bombieri to the third
author from 1996.  For $\Lambda_2$ replaced by $\Lambda_3$, a
corresponding (easier) statement was proven in \cite{FoMiPac}.
\par
Another potential application of this corollary (or rather of the
techniques used to prove it) is to reduce the value of the constant
$\omega$ in the following statement, which was first established by
Fouvry and Michel for $\omega=23$ in \cite{FMAnnals} subsequently
improved to $\omega=18$ by Sivak-Fischler~\cite{Siv1,Siv} and to
$\omega=15$ by Matom\"aki~\cite{Mat}:

\begin{theorem*} The sequence
$(\hypk_2(1;c))_{c\geq 1}$ changes sign infinitely often as $c$ varies
over squarefree moduli with at most $\omega$ prime factors.
\end{theorem*}

\subsection{Principle of the proof of Theorem~\ref{primesumthm}: the
  combinatorics of sums over primes}

We start from a general perspective before explaining what features
are specific to our case and what our new ingredients are. Suppose we
are given some oscillatory arithmetic function $K$, bounded by $1$ in
modulus, some smooth function $V$, compactly supported in
$]0,+\infty[$ and some $X\geq 2$; we wish to obtain non-trivial bounds
for the sum
$$
\sum_{n}\Lambda(n)K(n)V\Bigl(\frac{n}X\Bigr),
$$
where $\Lambda$ denotes the von Mangoldt function. 
\par
Using Heath-Brown's identity (see, e.g.,~\cite[Prop. 13.3]{KI}) and a
smooth partition of unity, this sum decomposes essentially into a
linear combination of sums of the shape
\begin{multline}\label{eq-hb-sum}
\sumsum_{m_1,\cdots,m_k}\alpha_1(m_1)\cdots\alpha_k(m_k)\sumsum_{n_1,\cdots,n_k}
V_1(n_1)\cdots V_k(n_k)\\
V\Bigl(\frac{m_1\cdots m_k n_1\cdots
  n_k}{X}\Bigr)K(m_1\cdots m_k n_1\cdots n_k)
\end{multline}
for some integral parameter $k\geq 2$, where the $\alpha_i(m)$ are
essentially bounded arithmetic functions supported in dyadic intervals
(say $[M_i/2,M_i]$) of short range (i.e. $M_i\leq X^{1/k}$), whereas
the $V_i(n)$ are smooth functions supported in dyadic intervals with
arbitrary range (say, $[N_i/2,N_i]$ with $N_i\in[1/2,2X]$), and where
$$
\prod_i M_iN_i\asymp X.
$$ 
\par
We refer to the $n_i$ as the ``smooth'' variables and the $m_i$ as the
``non-smooth'' variables, as one is usually unable to exploit the
specific shape of the functions $\alpha_i$, except for the fact that
they are supported in short ranges.

Depending on which estimates and methods are available to bound these
sums, according to the location of the point
$(M_1,\cdots,M_k,N_1,\cdots,N_k)$ in the $2k$-dimensional cube
$[1/2,2X]^{2k}$, it is useful to classify them into different (not
necessarily disjoint) categories, based on the number of ``long''
smooth variables which are available:
\begin{itemize}
\item[-] If there is one very long smooth variable, say $n_1$, one
  usually speaks of {\em sums of type} $I$, with the remaining (smooth
  and non-smooth variables) combined together into a single non-smooth
  variable, $m$, which means that the original sum~(\ref{eq-hb-sum}) may be
  written
$$
\sum_{m\asymp M}\beta_m\sum_{n_1\asymp N_1}
V_1(n_1)V\Bigl(\frac{mn_1}X\Bigr)K(mn_1).
$$
\item[-] If there are two relatively long smooth variables, say $n_1,
  n_2$, one speaks of sums of type $I_2$; after combining the
  remaining (smooth and non-smooth variables) into a single non-smooth
  variable, the sum can now be rewritten
$$
\sum_{m\asymp M}\alpha_m\sumsum_\stacksum{n_1\asymp N_1}{n_2\asymp
  N_2} V_1(n_1)V_2(n_2)V\Bigl(\frac{mn_1n_1}X\Bigr)K(mn_1n_2).
$$
\item[-] And if there are three relatively long smooth variables, say
  $n_1, n_2, n_3$, we will speak of sums of type $I_3$, and so on.
 \end{itemize}
\par
This classification appears more or less explicitly in the
work~\cite{FouvryAM} of Fouvry, in the context of the average
distribution of primes in arithmetic progressions to large moduli. The
implementation of this strategy depends on the possibility of dealing
with the sums of type $I_r$ for $r$ as large as possible, a question
which becomes increasingly difficult as $r$ increases, since the range
of the smooth variables decreases.\footnote{\ For instance, in
  \cite{FouvryAM}, it is shown that one could prove results on the
  distribution of primes in long arithmetic progressions on average,
  beyond the Bombieri-Vinogradov Theorem, if one could treat the
  corresponding sums of type $I_r$ for $r=1$, \ldots, $6$. Currently,
  the sums of type $I_1$, $I_2$ and $I_3$ can be handled \cite{FrIw}.}
All remaining sums belong then to the class of {\em sums of type}
$II$. The most direct treatment of these sums --there may be other
treatments available, depending on the original problem-- consists in
combining these (short) variables in subsets to form variables with
larger ranges, in order to obtain bilinear forms involving two
non-smooth variables of the type
$$
\sum_{m\asymp M}\sum_{n\asymp N}\alpha_m\beta_n K(mn),\quad\text{
  where } MN\asymp X.
$$
\par
One can then ``smoothen'' one of the variables, say $n$, by an
application of the Cauchy-Schwarz inequality, leading to a quadratic
form with coefficients with {\em multiplicative correlation sums} of
the function $K$, namely
$$
\dblsum_{m_1,m_2}\ov{\alpha_{m_1}}\alpha_{m_2}
\sum_{n}\ov{K(m_1n)}{K(m_2n)}.
$$
\par
Notice here that the fact that the original variables are rather short
actually helps, since it offers some flexibility in the ways they may
be combined to tailor the relative ranges of $M$ and $N$. This is the
strategy we will follow in this paper.

\subsection{Sums of type $I_2$}

We can now come to our specific situation and explain our new results
for sums over primes of trace weights.  
\par
We will give estimates for sums of type $I$, $I_2$ and $II$. In fact,
the starting point of this work is a very general estimate for sums of
type $I_2$ (two long smooth variables of approximately equal size)
when $K$ is a trace weight, which follows relatively easily from the
results of our earlier paper~\cite{FKM}. Indeed, using Mellin
inversion, the estimation of sums of type $I_2$ can be reduced to that
of sums of the shape
\begin{equation}\label{eq-eisenstein-twist}
\tsum_{V,X}(it,K)=\sum_{n}K(n)d_{it}(n)V\Bigl(\frac{n}X\Bigr)
\end{equation}
where $t\in\Rr$, and (for any $u\in\Cc$) we denote by
$$
d_{u}(n)=d_{-u}(n)=\sum_{ab=n}\Bigl(\frac{a}{b}\Bigr)^{u}
$$
the twisted divisor function.
\par
We observe that the arithmetic function $n\ra d_{it}(n)$ is (up to
suitable normalization) the Fourier coefficient of the non-holomorphic
unitary Eisenstein series
$$
E(z,s)=\frac12\sum_{(c,d)=1}\frac{y^s}{|cz+d|^{2s}},
$$
for $s=\frac{1}{2}+it$. The main result of our previous paper
(\cite[Thm 1.2]{FKM}) is a universal non-trivial bound for the
analogue of $\tsum_{V,X}(it,K)$ where $d_{it}(n)$ is replaced with the
Fourier coefficients of a classical cusp form (either holomorphic or
not). We will extend the proof to Eisenstein series, obtaining the
following result:

\begin{theorem}[Algebraic twists of Eisenstein series]
  \label{eisensteinsumthm}
  Let $K$ be an isotypic trace weight associated to the $\ell$-adic
  sheaf $\mcF$ modulo $p$. Let $V$ be a smooth function
  satisfying~\emph{(\ref{Vcond})} with parameter $Q\geq 1$. If $\mcF$
  is not geometrically trivial, then for any $X\geq 1$, we have
$$
\tsum_{V,X}(it,K)= \sum_{n}K(n)d_{it}(n)V\Bigl(\frac{n}X\Bigr)\ll
(1+|t|)^AQ X\Bigl(1+\frac{p}{X}\Bigr)^{1/2}p^{-\eta}
$$
for any $\eta< 1/8$ and some $A\geq 1$ possibly depending on
$\eta$. The implicit constant depends only on $\eta$, on the implicit
constants in~\emph{(\ref{Vcond})}, and polynomially on the conductor
of $\sheaf{F}$.
\end{theorem}

In fact, the proof of this theorem will be intertwined with the proof
of the following estimate on sums of type $I_2$:

\begin{theorem}[Type $I_2$ sums of trace weights]
\label{typeIsumthm} 
Let $K$ be an isotypic trace weight associated to the $\ell$-adic
sheaf $\mcF$ modulo $p$. Let $M,N, X\geq 1$ be parameters with
$X/4\leq MN\leq X$. Let $U$, $V$, $W$ be smooth functions satisfying
condition~\refs{Vcond} with respective parameters $Q_U,Q_V$ and $Q_W$,
all $\geq 1$. We then have
$$
\sum_{m,n}K(mn)\Bigl( \frac{m}n\Bigr)^{it}
U\Bigl(\frac{m}M\Bigr)V\Bigl(\frac{n}N\Bigr) W\Bigl(\frac{mn}{X}\Bigr)
\ll (1+|t|)^A(Q_U+Q_V)^{B} Q_W
X\Bigl(1+\frac{p}{X}\Bigr)^{1/2}p^{-\eta}
$$
for $t\in\Rr$ and for any $\eta< 1/8$ and some constants $A,B\geq 1$
depending on $\eta$ only. The implicit constant depends only on $\eta$,
on the implicit constants in~\emph{(\ref{Vcond})}, and polynomially on
the conductor of $\sheaf{F}$.
\end{theorem}

\begin{rem} 
  (1) Through the techniques of~\cite{FKM}, this result depends on
  deep results of algebraic geometry, including Deligne's general
  form of the Riemann Hypothesis over finite fields, and the theory of
  the $\ell$-adic Fourier transform of Deligne, Laumon and Katz.
\par
(2)  The Polya-Vinogradov method would yield a non trivial bound for the
  sum above as long as $\max(M,N)\gg p^{1/2}\log p$. Here we obtain
  non trivial estimates for $MN\gg p^{3/4+\eps}$ in particular when
  $M,N\gg p^{3/8+\eps}$.
\par
(3) From our point of view, the main innovation in this result, which
promises to have other applications, is that we handle the divisor
function in a fully automorphic manner, instead of attempting to use
its combinatorial structure as a Dirichlet convolution.
\end{rem}

\subsection{Sums of type $I$ and $II$}

Our second main result is a general estimate for sums of type $II$,
which gives non-trivial bounds, as long as one of the variables has
range slightly greater than $p^{1/2}\log p$ and the other has
non-trivial range. Precisely:

\begin{theorem}\label{typeIIsumthm} 
  Let $K$ be a \emph{non-exceptional} trace weight modulo $p$
  associated to an isotypic $\ell$-adic sheaf $\mcF$. Let $M,N\geq 1$
  be parameters, and let $(\alpha_m)_{m}$, $(\beta_n)_n$ be sequences
  supported on $[M/2,2M]$ and $[N/2,2N]$ respectively.
\par
\emph{(1)} We have
\begin{equation}
\label{typeIIeq}
\sumsum_\stacksum{m,n}{(m,p)=1}
\alpha_m\beta_n K(mn)
\ll\|\alpha\|\|\beta\|(MN)^{1/2}
\Bigl(\frac{1}{p^{1/4}}+\frac{1}{M^{1/2}}+
\frac{p^{1/4}\log^{1/2}p}{N^{1/2}}\Bigr),
\end{equation}
where
$$
\|\alpha\|^2=\sum_m|\alpha_m|^2,\ \|\beta\|^2=\sum_n|\beta_n|^2.
$$
\par
\emph{(2)} We have
\begin{equation}
\label{typeIeq}
\sum_{(m,p)=1}\alpha_m\sum_{n\leq N} K(mn)\ll
\Bigl(\sum_{m}|\alpha_m|\Bigr)N
\Bigl(\frac{1}{p^{1/2}}+\frac{p^{1/2}\log p}{N}\Bigr).
\end{equation}
\par
In both estimates, the implicit constants depend only, and at most
polynomially, on the conductor of $\mcF$.
\end{theorem}

This theorem constitutes a significant generalization of results like
\cite[Cor.~2.11]{MichelInv} or \cite[Prop.~1.3]{FMAnn}, which were
obtained for very specific weights (symmetric powers of Kloosterman
sums and additive characters of rational functions, respectively). The
main difference is that we do not require any knowledge of the
geometric monodromy group of $\mcF$. Instead, it turns out that we can
build on the same ideas used in~\cite{FKM} to handle algebraic twists
of cusp forms. A crucial role is played again by the $\ell$-adic
Fourier transform, and by a geometric invariant of $\mcF$ which we
introduced in~\cite{FKM}, namely its \emph{Fourier-M\"obius group},
which controls the correlation of the trace function of the Fourier
transform of $\mcF$ with its pullbacks under automorphisms of the
projective line. In fact, it is the intersection of this group with
the standard Borel subgroup of upper-triangular matrices of
$\PGL_2(\Fp)$ which we must understand, the essential point being that
this intersection is of size bounded in terms of the conductor of
$\mcF$ \emph{unless} $\mcF$ is exceptional. This is the origin of this
restriction in Theorem~\ref{typeIIsumthm}. It is rather remarkable
that the upper-triangular matrices in the Fourier-M\"obius group were
precisely those which do \emph{not} cause any difficulty in~\cite{FKM}
(hence in Theorem~\ref{typeIsumthm}).

\begin{remark} 
  For the purpose of Theorem~\ref{primesumthm}, it is indeed enough to
  handle sums of type $I_2$ and to deal will all others as sums of
  type $II$.  Other problems may require direct treatment of sums of
  type $I_r$ with $r\geq 3$ (see for instance the beautiful recent
  work of N. Pitt \cite{Pitt}). One might expect that this involves
  the theory of automorphic forms on $\GL_r$.
\end{remark}

\begin{remark} In~\cite{FMAnn}, the first and third authors obtained
  bounds which could be stronger than \refs{typeIIeq} and
  \refs{typeIeq}, in particular in ranges of $M$, $N$ which are
  shorter than the Polya-Vinogradov range $p^{1/2}$ (see
  \cite[Prop. 1.2 and Thm~1.4]{FMAnn}). These bounds were established
  only for very special weights associated to rank one sheaves
  (additive characters of specific rational functions). It is quite
  conceivable that these results remain valid for more general trace
  weights, and we hope to come back to this question in a later
  work. From our current level of understanding at least, it
  seems that, instead of the Fourier-M\"obius group (or in addition to
  it), we would need to involve more precise information on the
  underlying sheaf, for example concerning fine details of its
  ramification behavior, and/or its geometric monodromy group.
\end{remark}

\subsection{Acknowledgments}

Part of this work was done during the 60th birthday conference of
Roger Heath-Brown at Oxford. We would also like to thank the
organizers, Tim Browning, David Ellwood and Jonathan Pila for this
rich and pleasant week. Roger Heath-Brown has obtained fundamental
results in the theory of the primes; it is no surprise that, once
more, the celebrated ``Heath-Brown identity'' makes a crucial
appearance in the present work. 
\par
This paper has benefited from discussions and comments from Jean
Bourgain, Satadal Ganguly, Paul Nelson, Richard Pink, Ng\^o Bao
Ch\^au, Peter Sarnak, Igor Shparlinski, Akshay Venkatesh, and
 Zhiwei Yun and it is a pleasure to thank them for their input.
 
 Finally, we are thankful to the referees for their careful reading
  of the manuscript.

\section{Algebraic twists of Eisenstein series and sums of type $I_2$}

In this section, we will prove Theorem~\ref{eisensteinsumthm} and
Theorem~\ref{typeIsumthm} simultaneously. Indeed, the two results are
very closely related, as we will first clarify.  
\par
Let $M,N, X\geq 1$ with $X/4\leq MN\leq 4X$. Let $t\in\Rr$ be given,
as well as three smooth functions $U$, $V$, $W$ satisfying
\refs{Vcond} with respective parameters $Q_U$, $Q_V$, $Q_W$, all $\geq
1$. We package these parameters by denoting
$$
\uple{P}=(U,V,W,M,N,X),
$$ 
and we denote
$$
\tsum_{\uple{P}}(it,K) =\sum_{m,n}K(mn)\Bigl(\frac{m}n\Bigr)^{it}
U\Bigl(\frac{m}M\Bigr)V\Bigl(\frac{n}N\Bigr)
W\Bigl(\frac{mn}{X}\Bigr),
$$
which is the sum involved in Theorem~\ref{typeIsumthm}. For later use,
we state formally here the trivial bound
\begin{equation}\label{eq-trivial-bound}
\tsum_{V,X}(it,K)\ll X(\log X)
\end{equation}
for the sums $\tsum_{V,X}(it,K)$.
\par
We start with a lemma relating the sums of type $\tsum_{V,X}(\cdot,K)$
and $\tsum_{\uple{P}}(\cdot,K)$.

\begin{lemma}\label{lm-relations} We adopt the above notations and for $s\in {\mathbb C}$ and $x >0$, let 
$$
W_{s}(x):=W(x)x^{-s}.
$$
  \emph{(1)} For every $\eps >0$, there exists $C=C(\eps)$, such that we have
$$
\tsum_{\uple{P}}(it,K)\ll_\eps  (Q_U+Q_V)^C+ \iint_{|t_1|,|t_2|\leq
  X^{\eps}}{ |\tsum_{W_{t_1},X}(it_2+it,K)|dt_1dt_2 }.
$$
 \emph{(2)} For every $\eps >0$, one has 
$$
\tsum_{V,X}(it,K)\ll_\eps  X^{\eps}\max_{\uple{P}=(U_1,V_1,V,M,N,X)}
|\tsum_{\uple{P}}(it,K)|,
$$
where $\uple{P}$ runs over parameters $(U_1,V_1,V,M,N,X)$ as above
with $Q_{U_1}=Q_{V_1}=1$.
\end{lemma}

\begin{proof}
  (1) Denote by $\hat U$ and $\hat V$ the Mellin transforms of the
  smooth functions $U$ and $V$. These are entire functions, which
  satisfy
\begin{equation}
\label{UVdecay}
\hat U(s),\hat V(s)\ll \Bigl(\frac{Q_U+Q_V}{1+|s|}\Bigr)^k,
\end{equation}
for any $k\geq 0$, where the implicit constants depend on $k$, $\Re s$
and the implicit constants in~(\ref{Vcond}).
\par
We then have
\begin{align*}
  \tsum_{\uple{P}}(it,K) &= \frac{1}{(2i\pi)^2}
  \int_{(0)}\int_{(0)}\hat{U}(u)\hat{V}(v)\mathcal{T}_W(u,v)N^uM^vdudv
\end{align*}
by Mellin inversion, where
$$
\mathcal{T}_W(u,v)=
\sum_{m,n\geq 1}{K(mn)m^{it-u}n^{-it-v}W\Bigl(\frac{mn}{X}\Bigr)}.
$$
\par
This sum can be expressed as a twist of Eisenstein
series~(\ref{eq-eisenstein-twist}), namely
$$
\mathcal{T}_W(u,v)=X^{-\theta_1}\tsum_{W_{\theta_1},X}(\theta_2+it,K),
$$
where 
\begin{gather*}
  \theta_1=\frac{u+v}{2},\quad\quad \theta_2=\frac{-u+v}{2}.
\end{gather*}
\par
Thus, by a change of variable, we get
$$
\tsum_{\uple{P}}(it,K)=
\frac{2}{(2i\pi)^2}
\int_{(0)}\int_{(0)}
\hat{U}(\theta_1-\theta_2)
\hat{V}(\theta_1+\theta_2)
 \Bigl(\frac{M}{N}\Bigr)^{\theta_2}
\Bigl(\frac{MN}{X}\Bigr)^{\theta_1}
\tsum_{W_{\theta_1},X}(\theta_2+it,K)d\theta_2d\theta_1.
$$
\par
The function $W_{\theta_1}$ is smooth and compactly supported on
$[1/2,2]$. For $\Re \theta_1=0$, it satisfies~(\ref{Vcond}) with parameter
\begin{equation}\label{eq-qw1}
Q(\theta_1)\ll Q_W+|\theta_1|,
\end{equation}
where the implicit constant is absolute. 
\par
Using~(\ref{UVdecay}) for $k$ large enough, and the trivial
bound~(\ref{eq-trivial-bound}), the contribution to this double
integral of the region where $|\theta_1|\geq X^{\eps}$ or
$|\theta_2|\geq X^{\eps}$ is
$$
\ll (Q_U+Q_V)^C
$$
for some $C\geq 0$ depending only on $\eps$, which concludes the
proof.
\par
(2) By a dyadic partition of unity (using Lemma \ref{dyadic} below),
and taking into account the support condition, we can decompose
$\tsum_{V,X}(it,K)$ into $O(\log X)$ sums of the shape
$$
\tsum_{\uple{P}}(it,K)
$$
where
$$
\uple{P}=(U_1,V_1,V,M,N,X)
$$
with $X/4\leq MN\leq 4X$, and furthermore the functions $U_1$, $V_1$
satisfy condition \refs{Vcond} with parameters $Q_{U_1}=Q_{V_1}=1$.
The result is then immediate.
\end{proof}

\subsection{A simple bound}

We start with the following simple ``convexity'' bound for the
Eisenstein twists $\tsum_{V,X}(it,K)$, which is useful for $X\geq p$,
and will indeed imply both Theorem~\ref{typeIsumthm} and
Theorem~\ref{eisensteinsumthm} for $X\geq p^{5/4+\eps}$.

\begin{lemma}\label{convexlem} 
  With the notation and assumptions of Theorem \ref{eisensteinsumthm},
  we have for any $\eps>0$,
\begin{equation}
\label{convexeq}
\tsum_{V,X}(it,K)\ll 
\bigl(pQX(1+|t|)\bigr)^\eps 
(1+|t|)^{1/2}QX\Bigl(\frac{1}{p}+\frac{p}{X}\Bigr)^{1/2},
\end{equation}
where the implicit constant depends on $\eps$ and polynomially on
$\cond(\mcF)$. 
\end{lemma}

\begin{proof}
  This is relatively standard, so we will be fairly brief: the idea is
  to use periodicity of $K$ and to represent it in terms of Dirichlet
  characters, reducing then to easy estimates for moments of Dirichlet
  $L$-functions.  
\par
First of all,  the contribution to $\tsum_{V,X}(it,K)$ of
the integers $n$ divisible by $p$ is
$$
\sum_{n\equiv 0\mods{p}}K(0)d_{it}(n)V\Bigl(\frac{n}X\Bigr)
\ll_{\cond(\mcF)} p^{-1}X\log X.
$$ 
\par
Next, for $(n,p)=1$, we can write
$$
K(n)=\frac{1}{(p-1)^{1/2}}\sum_{\chi}\tilde K(\chi)\chi(n)
$$
where $\chi$ runs over the Dirichlet characters modulo $p$ and
$$
\tilde K(\chi)=\frac{1}{(p-1)^{1/2}}\sum_{m\in\Fpt}K(m)\ov\chi(m)
$$
is the finite-field Mellin transform of $K$. Thus we get
$$
\sum_{(n,p)=1}K(n)d_{it}(n)V\Bigl(\frac{n}X\Bigr)
=\frac{1}{(p-1)^{1/2}}\sum_{\chi}\tilde
K(\chi)\sum_{n}\chi(n)d_{it}(n)V\Bigl(\frac{n}{X}\Bigr).
$$
\par
The contribution of the trivial character $\chi_0$ to this sum
is estimated by
$$
\frac{1}{(p-1)^{1/2}}\tilde K(\chi_0)
\sum_{n}\chi_0(n)d_{it}(n)V\Bigl(\frac{n}{X}\Bigr)\ll_{\cond(\mcF)}
p^{-1/2}X\log X
$$
(indeed, since $K$ is a trace weight, we have
$$
(p-1)^{1/2}\tilde K(\chi_0)=\sum_{m\in\Fp}K(m)-K(0)=p^{1/2}\hat
K(0)-K(0)\ll_{\cond(\mcF)} p^{1/2},
$$ 
where $\hat K$ denotes the unitarily normalized Fourier transform of
$K$, so that $-\hat{K}$ is also an isotypic trace weight, associated
to a sheaf with conductor bounded in terms of that of $\sheaf{F}$
only, by the properties of the Fourier transform of $\ell$-adic
sheaves, as explained in~\cite[\S 1.4, Prop. 8.2]{FKM}.)
\par
For $\chi$ non-trivial, denoting by $\hat{V}(s)$ the Mellin transform
of $V$, we have
$$
\sum_{\chi\not=\chi_0}\tilde
K(\chi)\sum_{n}\chi(n)d_{it}(n)V\Bigl(\frac{n}{X}\Bigr)=
\intc_{\Re
  s=1/2}\sum_{\chi\not=\chi_0}\tilde K(\chi)
L(\chi,s+it)L(\chi,s-it)\hat{V}(s)X^sds,
$$ 
by a standard application of Mellin inversion and a contour shift.
\par
From~(\ref{Vcond}), we get 
\begin{equation}\label{eq-117}
\hat{V}(s)\ll_j \Bigl(\frac Q{1+|s|}\Bigr)^j
\end{equation}
for all $j\geq 0$, where the implicit constant depends on $j$. Now,
for any fixed $\eps>0$, let $S=Q$, and split the $s$-integral into
$$
\intc_{\Re s=1/2}\cdots=\intc_\stacksum{\Re s=1/2}{|\Im s|\leq
  S}\cdots+\intc_\stacksum{\Re s=1/2}{|\Im s|> S}\cdots = I_1+I_2,
$$
say. To handle $I_1$, we apply Cauchy's inequality to obtain
$$
I_1^2\ll \Bigl\{XS\sum_{\chi}{|\tilde{K}(\chi)|^2}\Bigr\}
\times \int_\stacksum{\Re s=1/2}{|\Im(s)|\leq S}{
\sum_{\chi\not=\chi_0}
|
L(\chi,s-it)L(\chi,s+it)
|^2|ds|
}.
$$
\par
By the Parseval identity, the first factor on the right-hand side is
$\ll pXS$. For the second factor, we apply the approximate functional
equation to bound the product $L(\chi,s-it)L(\chi,s+it)$ by sums of
the shape 
$$
\sum_{n}\frac{\chi(n)d_{it}(n)}{n^s}W\Bigl(\frac{n}{N}\Bigr)
$$ 
for $W$ rapidly decreasing and $N\ll p(1+S+|t|)$ (see,
e.g.,~\cite[Th. 5.3]{KI}).  By a hybrid-large sieve estimate
(see~\cite[Th. 6.4]{montgomery}, compare with~\cite[Th. 7.34]{KI}), we
can get the Lindel\"of conjecture on average for the integral and sum,
and therefore derive
$$
 \int_\stacksum{\Re s=1/2}{|\Im(s)|\leq S}{
\sum_{\chi\not=\chi_0}
|
L(\chi,s-it)L(\chi,s+it)
|^2|ds|
}\ll_\eps (pS(1+|t|))^{1+\eps}.
$$
\par
To estimate $I_2$, we split the interval of integration $|\Im(s)|\geq
S$ into segments $2^\nu S\leq |\Im(s)| \leq 2^{\nu+1}S$ for $\nu\geq
0$. Applying to each segment the last inequality just above with $S$
replaced by $2^{\nu} S$ together with the rapid decay of $\hat{V}(s)$
(see~(\ref{eq-117})), we obtain the same type of bound for $I_2$ after
summation over $\nu$. 
\end{proof}

For $X\geq p^{3/2}$, for instance, the bound in this lemma is stronger
than the one claimed in Theorem~\ref{eisensteinsumthm}. Using
Lemma~\ref{lm-relations} (1), we then also deduce
Theorem~\ref{typeIsumthm} in this range. We will therefore assume for
the remainder of this section that $X\leq p^{3/2}$. Similarly,
comparing the bounds of Theorems~\ref{eisensteinsumthm} and
\ref{typeIsumthm} with the trivial bounds 
$$
\tsum_{V,X}(it,K)\ll X\log X,\quad 
\quad
\tsum_{\mathbf P}(it,K)\ll X\log X,
$$ 
we may assume that $X\geq p^{3/4}$.

\subsection{Spectral theory and amplification} 

The most important ingredient in the proof of
Theorems~\ref{eisensteinsumthm} and \ref{typeIsumthm}, is the
following lemma, which is proved with the methods of~\cite{FKM}, based
on the amplification method and Kuznetsov's formula. It is an averaged
version of a bound for the amplified second moment of the sums
$\tsum_{V,X}(it,K)$. Recall that $K$ is an isotypic trace weight.
\par
For $\tau\in\Rr$, $L\geq 1$ and $u\in\Cc$, let
$$
B_{i\tau}(u)=\sum_\stacksum{\ell\leq 2L}{\ell\text{ prime}}
\mathrm{sign}(d_{i\tau}(\ell))d_{u}(\ell),
$$
which is the amplifier (of length $2L$) adapted to the Eisenstein
series $E(z,1/2+i\tau)$.

\begin{lemma}\label{lm-average-bound}
  For any $\eps>0$ there exists $b=b(\eps)\geq 0$ such that
\begin{equation}
\label{eq-average-bound}
\int_\Rr\min(|t|^{2},|t|^{-2-2b})|B_{i\tau}(it)\tsum_{V,X}(it,K)|^2dt\ll_\eps
p^\eps\bigl(pLQX+p^{1/2}L^3QX(X/p+Q)^2\bigr),
\end{equation}
provided
$$
p^{\eps}LQ<p^{1/4},\quad\quad 1\leq L\leq X.
$$
\end{lemma}

\begin{proof}
  As in the cuspidal case in~\cite{FKM}, we use the amplification
  method and the Kuznetsov formula, exploiting the fact that, for any
  given $\tau\in\Rr$, the Eisenstein series
$$
\frac{1}{(p+1)^{1/2}}E(z,1/2+i\tau)
$$ 
occurs in the continuous spectrum of Hecke eigenforms of level
$p$. More precisely, we have the Fourier expansion
$$
E(z,1/2+it)=y^{1/2+it}+\frac{\theta(1/2-it)}{\theta(1/2+it)}y^{1/2-it}+
\frac{1}{\theta(1/2+it)}\sum_{n\not=0}d_{it}(|n|)|n|^{-1/2}W_{it}(4\pi
|n|y)e(n x),
$$
where
$$
\theta(s)=\pi^{-s}\Gamma(s)\zeta(2s),
$$
and 
\begin{equation}
\label{Wit}
W_{it}(y)=
\frac{e^{-y/2}}{\Gamma(it+\frac12)}
\int_0^\infty e^{-x}x^{it-1/2}\Bigl(1+\frac{x}y\Bigr)^{it-1/2}dx
\end{equation}
denotes the Whittaker function (see for instance
\cite[(3.29)]{IwaIntro}).
\par
We assume that the condition of the lemma are met. Using the notation
of~\cite[Section 4]{FKM} and taking there $P=Xp^{-1}$, we obtain as
in~\cite[Prop. 4.1, (4.10)]{FKM} (with the parameter $M$ given by
applying~\cite[Th. 1.14]{FKM} to the trace weight $K$, so that
$M=aN^s$ for some absolute constants $a>0$ and $s\geq 1$) the bound
\begin{multline}
\label{meansquarebound}
\frac{1}{p+1} \int_\Rr\tilde\phi_{a,b}(t)\frac{1}{\cosh(\pi
  t)|\theta(1/2+it)|^2} |B_{i\tau}(it)|^2|\tsum_{V,X}(it,K)|^2dt
\\
\leq M(L)-2\sum_{k>a-b }\dot\phi(k)(k-1)M(L;k)\ll p^\eps\Bigl(
LQX+\frac{L^3QX}{p^{1/2}}\Bigl(\frac Xp+Q\Bigr)^2\Bigr)
\end{multline}
for any $\eps>0$, where $2\leq b<a$ are odd integers depending on
$\eps$ and $\tilde\phi(t)=\tilde{\phi}_{a, b}(t)$ denotes a positive
function such that
$$
\tilde\phi(t)\asymp_{a, b}\, (1+|t|)^{ - 2b-2}.
$$
\par
We then obtain the desired average estimate from this, using
Stirling's formula and the bound 
$$
\zeta(1+2it)\ll
\frac{1}{|t|}+\log(1+|t|).
$$ 
\end{proof}

In order to apply this, we need some lower bound for the amplifier
$B_{i\tau}(it)$.  The following lemma gets a suitable bound for $t$
close enough to $\tau$.

\begin{lemma}\label{Flemma}	
For $L$ large enough, we have
$$
B_{i\tau}(i t)\gg\frac{L}{\log^6 L},
$$
uniformly for $t$ and $\tau\in\Rr$ satisfying 
$$
|t-\tau|\ll \frac{1}{\log^7 L},\quad\text{ and } \quad |\tau|\leq
L^\frac{1}{3}.
$$
\end{lemma}

\begin{proof}
  We observe first that for any prime $\ell\leq 2L$ and $|t-\tau|\ll
  \log^{-7} L$, we have
\begin{align*}
  |B_{i\tau}(it)-B_{i\tau}(i\tau)|&\leq
  \sum_\stacksum{\ell\leq 2L}{\ell\ \text{prime}}|d_{i\tau}(\ell)-d_{it}(\ell)|\\
  &=2\sum_\stacksum{\ell\leq 2L}{\ell\ \text{prime}}
  |\cos(\tau\log\ell)-\cos(t\log \ell)|\\
  &\leq 2|t-\tau|\sum_\stacksum{\ell\leq 2L}{\ell\ \text{prime}}\log
  \ell \ll \frac{L}{\log^7 L},
\end{align*}
and hence it suffices to prove the lower bound  for $t=\tau$.
\par
Furthermore, we may clearly assume that $\tau>0$ (by parity) and that
$L\geq 3$. We then have
$$
B_{i\tau}(i\tau)= \sum_{\stacksum{\ell\leq 2L}{\ell\text{ prime}}}
{\mathrm{sign}(d_{i\tau}(\ell))d_{i\tau}(\ell)}=
\sum_{\stacksum{\ell\leq 2L}{\ell\text{ prime}}} |d_{i\tau}(\ell)|
=2\sum_{\stacksum{\ell\leq 2L}{\ell\text{ prime}}}
|\cos(\tau\log\ell)|,
$$
and since $|\cos(\tau\log \ell)|\leq 1$ it is enough to prove that
$$
\sum_{\ell\sim L}\cos^2(\tau\log \ell)
\gg \frac{L}{\log^6 L}
$$
(where $\ell$ ranges over primes $L<\ell\leq 2L$) under the assumption
of the lemma. We do this by finding suitable sub-intervals where
$\tau\log \ell$ is sufficiently far away from $\pi/2$ modulo
$\pi\Zz$. 
\par
Consider the function 
$$
g(x)=\tau\log x
$$ 
for $x\in[L,2L]$. It is non-decreasing and satisfies
$$
g(2L)-g(L)=\tau\log 2,\
g'(x)\in\left[\frac{\tau}{2L},\frac{\tau}{L}\right] \text{ for } x\in
[L,2L].
$$
\par
In particular, if $\tau\log 2\geq 2\pi$, the preimage
$g^{-1}([-\pi/4,\pi/4]+\pi\Zz)$ intersects $[L,2L]$ in $\gg \tau$
intervals of length $\geq \frac{\pi L}{2\tau}$.  From Huxley's Theorem
on primes in short intervals (see, e.g.,~\cite[Th. 10.4, Th. 10.5]{KI}
and note that any of the variants of~\cite[Th. 10.5]{KI}, going back
to Hoheisel, would be enough) the number of primes in any such
interval is $\gg \frac{L}{\tau\log L}$, provided $\tau\leq
L^{5/12-\eps}$ for some fixed $\eps>0$ and $L$ is large
enough. Therefore (taking $\eps=1/12$ and summing over these
intervals) we obtain
$$
|B_{i\tau}(i\tau)|\gg \frac{L}{\log L}
$$
provided $2\pi/\log 2\leq \tau\leq L^{1/3}$.
\par
At the other extreme, if $0\leq \tau \leq \frac{1}{100\log L}$, we
have
$$
\cos^2(\tau\log \ell)\geq \cos^2\Bigl(\frac{1}{50}\Bigr)\
$$
for every $L\leq \ell\leq 2L$, and hence $B_{i\tau}(i\tau)\gg L/\log
L$ also in that case.
\par
Suppose now that 
$$
\frac{1}{100\log L}\leq\tau\leq \frac{2\pi}{\log 2}.
$$
\par
In that case, $g([L,2L])$ is an interval of length at least
$1/(200\log L)$. It is then easy to see (the worst case is when the
interval is symmetric around $\pi/2+k\pi$ for some integer $k$) that there
exists $x_0\in [L,2L]$ such that
$$
\cos^2(\tau\log(x_0))\geq \frac{1}{2\cdot 400^2\log^2 L}.
$$
\par
Using the prime number theorem with sufficiently precise error term,
we know that the interval $[L,2L]\cap[x_0-L(\log L)^{-3},x_0+L(\log
L)^{-3}]$ contains at least $\gg L/(\log L)^4$ primes, and since
$$
|\cos^2(\tau\log (\ell))-\cos^2(\tau\log (x_0))|\ll |\log(\ell/x_0)|\ll
(\log L)^{-3}
$$
for these primes, we have
$$
\cos^2(\tau\log (\ell))\gg \log^{-2} L,
$$
and therefore by summing over these $\ell$ we get
$$
B_{i\tau}(i\tau)\gg \frac{L}{\log^{6} L}
$$
in that last case, which concludes the proof.
\end{proof}

This lemma and the average bound~\refs{eq-average-bound} allow us to
deduce a first good upper-bound for the twists of Eisenstein series,
averaged in rather short intervals. It will be convenient for later
purposes to introduce the notation
\begin{gather}\label{eq-lmax-tau}
  I(\tau,p)=\{t\in \Rr\,\mid\, |t-\tau|\leq \log^{-7}p\},\quad
  \lmax(\tau,p)=\max_{t\in I(\tau,p)}|\tsum_{V,X}(it,K)|,\quad\\
  M(Q,X)=QX\Bigl(1+\frac{p}{X}\Bigr)^{1/2}p^{-1/8},
\end{gather}
so that, for instance, Theorem~\ref{eisensteinsumthm} claims that
$$
\tsum_{V,X}(it,K)\ll p^{\eps}(1+|t|)^A M(Q,X)
$$
for any $\eps>0$ and $A\geq 1$ depending on $\eps$.
\par
For our next result, we recall that we work under the assumptions
$p^{3/4}<X<p^{3/2}$ and $1\leq Q \leq p$. We then have:

\begin{proposition}\label{shortaveragebound}
For any $\eps>0$, there exists $B\geq 5$, depending only on $\eps$,
such that for any $\tau\in\Rr$ we have
\begin{equation}\label{eq-first-average}
  \int_{I(\tau,p)}\min(|t|^2,1)|\tsum_{V,X}(i
  t,K)|^2dt\ll_\eps   p^\eps(1+|\tau|)^{B}M(Q,X)^2,
\end{equation}
where the implied constant depends only on $\eps$.
\end{proposition}

\begin{proof}
Let
$$
L=\frac{p^{1/4-\eps}}{X/p+Q}
$$
as in~\cite[(4.2)]{FKM}. If $L\ll 1$ or if $|\tau|>L^{1/3}$, the
trivial bound~(\ref{eq-trivial-bound}) or the convexity bound
\refs{convexeq} yield stronger results than~(\ref{eq-first-average})
(since $B\geq 5$).  With this definition of $L$, and the above
reduction, we may apply Lemmas~\ref{lm-average-bound} and~\ref{Flemma} 
to the remaining cases. We obtain
\begin{align*}
  \int_{|t-\tau|\leq \log^{-7}p}\min(|t|^2,|t|^{-2-2b})
  |\tsum_{V,X}(it,K)|^2dt &\ll_\eps
  p^\eps\Bigl(\frac{pXQ}{L}+p^{1/2}XQL\Bigl(\frac{X}{p}+Q\Bigr)^2
  \Bigr)\\
  &\ll p^\eps Q^2X^2\Bigl(1+\frac{p}{X}\Bigr)p^{-1/4},
\end{align*}
for some $b$ depending on $\eps$. It only remains to note
inequality 
$$ 
\min (| t |^2, 1) \ll (1+| \tau|)^{2b+2} \min (| t|^2, | t |^{-2b-2}),
$$
for $t \in I(\tau,p)$, and to choose $B(\eps)=\max (5, 2 b
+2)$, in order to complete the proof of~(\ref{eq-first-average}) in
all cases.
\end{proof}

The remaining objective is to derive a pointwise bound for
$\tsum_{V,X}(it,K)$, and to do so we must relax the zero of order $2$
of the weight $\min(|t|^2,1)$ at the origin (for similar issues with
estimates of $L$-functions, see e.g.~\cite{blomer}; we could use
similar methods, but at the expense of expressing our sums in terms of
$L$-functions, and instead we use them directly, and resort to an
iterative argument.)
\par
The basic mechanism is the following consequence of
Proposition~\ref{shortaveragebound}:

\begin{corollary}\label{corshortaverage}
  For any $\tau\in\Rr$ and $\eps>0$, with notation as above, we have
$$
\int_{I(\tau,p)}|\tsum_{V,X}(i t,K)|dt\ll_\eps  
\begin{cases}
p^\eps(1+|\tau|)^{B}M(Q,X)&\text{ if } |\tau|\geq 1,\\
p^\eps \lmax(\tau,p)^{1/3}M(Q,X)^{2/3}&\text{ if } |\tau|\leq 1.
\end{cases}
$$
\end{corollary}

\begin{proof} 
If $|\tau|\geq 1$, we just apply the Cauchy-Schwarz inequality to get
\begin{align*}
\int_{I(\tau,p)}|\tsum_{V,X}(i t,K)|dt
&\leq \Bigl(\int_{I(\tau,p)}|\tsum_{V,X}(i t,K)|^2dt\Bigr)^{1/2}
\Bigl(\int_{I(\tau,p)}dt\Bigr)^{1/2}\\
&\ll \Bigl(\int_{I(\tau,p)}\min(|t|^2,1)|\tsum_{V,X}(i
t,K)|^2dt\Bigr)^{1/2}
\ll p^{\eps}(1+|\tau|)^BM(Q,X)
\end{align*}
by Proposition~\ref{shortaveragebound}.
\par
Now assume $|\tau|<1$. Let $0<\alpha<1/3$ be some parameter. By
H\"older's inequality, we get
\begin{align*}
\int_{I(\tau,p)}|\tsum_{V,X}(i t,K)|dt&\leq
\lmax(\tau,p)^{1-2\alpha}
\int_{I(\tau,p)}|\tsum_{V,X}(i t,K)|^{2\alpha}dt
\\
&\leq 
\lmax(\tau,p)^{1-2\alpha}
\Bigl(\int_{I(\tau,p)}|t|^2|\tsum_{V,X}(i t,K)|^{2}dt\Bigr)^{\alpha}
\Bigl(\int_{0}^2|t|^{-2\alpha/(1-\alpha)}dt\Bigr)^{1-\alpha}
\\
&\ll_{\eps,\alpha} p^\eps \lmax(\tau,p)^{1-2\alpha}M(Q,X)^{2\alpha}\\
&\ll_{\eps,\alpha} p^\eps \lmax(\tau,p)^{1-2\alpha}M(Q,X)^{2/3},
\end{align*}
(since $2\alpha/(1-\alpha)<1$ and $M(Q,X)\geq 1$). By the trivial
bound~(\ref{eq-trivial-bound}), we have
$$
\lmax(\tau,p)^{1-2\alpha}\ll \lmax(\tau,p)^{1/3}(X\log X)^{1-2\alpha-1/3},
$$ 
and we conclude by taking $\alpha=1/3-\eps$.
\end{proof}

\subsection{An iterative bound}

The following lemma establishes an iterative bound for Eisenstein
twists. 

\begin{lemma}
Assume that $\beta>0$ is such that
\begin{equation}\label{eq-iter-assumption}
\tsum_{V,X}(it,K)\ll p^{\eps}(1+|t|)^AX^{\beta}M(Q,X)^{1-\beta}
\end{equation}
for $X\leq p^{3/2}$, any $\eps>0$, and some $A\geq 1$ depending on
$\eps$. Then for any $\eps>0$, we have
\begin{align}
  \tsum_{V,X}(it,K)&\ll p^{\eps}(1+|t|)^{A_1}X^{\beta/3}M(Q,X)^{1-\beta/3}
  \label{eq-1}\\
  \tsum_{\uple{P}}(it,K)&\ll p^{\eps}
  (Q_U+Q_V)^B(1+|t|)^{A_1}X^{\beta/3}M(Q_W,X)^{1-\beta/3}\label{eq-2}
\end{align}
for $A_1$, $B\geq 1$ depending on $\eps$.
\end{lemma}

\begin{proof}
  Using Lemma~\ref{lm-relations} (1), we first use the assumption to
  estimate $\tsum_{\uple{P}}(it,K)$. For each $t_1$, we split the
  integral over $|t_2|\leq p^{\eps}$ into $\ll p^{\eps}$ integrals
  over intervals of length $\log^{-7}p$. For an interval $I$ with
  center at $\tau$ with $|\tau|\leq 1$, the integral is bounded by
$$
\ll p^{\eps}\mathcal{M}^{1/3}M(Q_W+|t_1|,X)^{2/3}
$$
by Corollary~\ref{corshortaverage} applied to $W_{t_1}$
(see~(\ref{eq-qw1})), where
$$
\mathcal{M}=\max_{t\in I} |\tsum_{W_{t_1},X}(it,K)|\ll
p^{\eps} X^{\beta}M(Q_W+|t_1|,X)^{1-\beta}
$$
by~(\ref{eq-iter-assumption}) and~(\ref{eq-qw1}). Thus each such
integral is
$$
\ll p^{\eps} X^{\beta/3}M(Q_W+|t_1|,X)^{1-\beta/3}.
$$
\par
For intervals centered at $\tau$ with $1\leq |\tau|\leq p^{\eps}$, we
obtain the bound $\ll p^{\eps}(1+|\tau|)^AM(Q_W+|t_1|,X)$, which is
better, and integrating over $|t_1|\leq p^{\eps}$, we get~(\ref{eq-2})
(note that $Q\mapsto M(Q,X)$ is linear).
\par
Now, applying Lemma~\ref{lm-relations} (2), we immediately
deduce~(\ref{eq-1}). 
\end{proof}

We are now done: for $p^{3/4}\leq X\leq p^{3/2}$, we can start
applying this lemma with $\beta=1$ by the trivial
bound~(\ref{eq-trivial-bound}). We deduce that, for any integer $k\geq
1$, we have
\begin{align*}
  \tsum_{V,X}(it,K)&\ll p^{\eps}(1+|t|)^AX^{3^{-k}}M(Q,X)^{1-3^{-k}}
  \\
  \tsum_{\uple{P}}(it,K)&\ll p^{\eps}
  (Q_U+Q_V)^B(1+|t|)^AX^{3^{-k}}M(Q_W,X)^{1-3^{-k}}.
\end{align*}
\par
Since
\begin{align*}
X^{3^{-k}}M(Q,X)^{1-3^{-k}}&=
XQ^{1-3^{-k}}\Bigl(1+\frac{p}{X}\Bigr)^{(1-3^{-k})/2}p^{-(1-3^{-k})/8}\\
&\leq XQ(1+p/X)^{1/2}p^{-1/8}p^{3^{-k}/8},
\end{align*}
we therefore obtain Theorems~\ref{eisensteinsumthm}
and~\ref{typeIsumthm} by taking $k$ large enough.

\section{Estimating sums of type $II$}

In this section we prove Theorem \ref{typeIIsumthm}. We will leave the
proof of the simpler bound~(\ref{typeIeq}) to the reader, and
consider~(\ref{typeIIeq}), proceeding along classical lines. Denoting
$$
T=\sumsum_\stacksum{m,n}{(m,p)=1}
\alpha_m\beta_n K(mn)
$$
the bilinear form to estimate, we apply Cauchy's inequality and deduce
that
\begin{equation}\label{eq-norm-bil}
|T|^2\leq \|\beta\|^2 \dblsum_{\stacksum{M/2\leq m_1,m_2\leq 2M}{p\nmid
    m_1m_2}} \ov{\alpha_{m_1}}\alpha_{m_2}\sum_{N/2\leq n\leq
  2N}\ov{K(m_1n)}K(m_2n).
\end{equation}
\par
The inner correlation coefficients are then treated by completion
(i.e., by the Polya-Vinogradov method), which gives
\begin{equation}\label{eq-correl}
\sum_{N/2\leq n\leq 2N}\ov{K(m_1n)}K(m_2n) \ll
\frac{N}p|\wwd(m_1,m_2,0,K)|+ \sum_{0<|h|\leq
  p/2}\min\Bigl(\frac{1}{|h|},\frac{N}{p}\Bigr) |\wwd(m_1,m_2,h,K)|
\end{equation}
where
$$
\wwd(m_1,m_2,h,K)=\sum_{z\in\Fp}\ov{K(m_1z)}K(m_2
z)e\Bigl(\frac{hz}p\Bigr),
$$
a sum which satisfies the relation
$$
\wwd(m_1,m_2,h,K)=\wwd(m_1/m_2,1,h/m_2,K).
$$
\par
For a trace weight, we have the trivial bound
$$
|\wwd(m_1,m_2,h,K)|\leq \cond(\mcF)^2 p,
$$
but this is not sharp in most cases. In fact, the crucial point is to
show that for most parameters $(m_1,m_2,h)$, we have a better estimate
with square-root cancellation. We provide such a result in
Theorem~\ref{th-bound-bad} in Section~\ref{sec-prelim}, building on
our earlier work in~\cite{FKM}.

\begin{proposition}[Paucity of large
  correlations]\label{correlationprop}
  Let $K$ be an irreducible trace weight modulo $p$ which is not
  $p$-exceptional, associated to the sheaf $\mcF$. Then there exists
  $C\geq 1$, $D\geq 0$, depending only polynomially on
  $\cond(\sheaf{F})$, such that
$$
|\wwd(m,1,h,K)|\leq Cp^{1/2}
$$
for every pair $(m,h)\in\Fpt\times\Fp$ except for those in a set of
pairs of cardinality at most $D$.
\end{proposition}

After inserting~(\ref{eq-correl}) in~(\ref{eq-norm-bil}), the
contribution of all triples $(m_1,m_2,h)$ for which
$$
|\wwd(m_1,m_2,h,K)|\leq Cp^{1/2}
$$
is at most
$$
\ll \|\alpha\|^2\|\beta\|^2\Bigl(\frac{MN}{p^{1/2}}+Mp^{1/2}\log p\Bigr).
$$
\par
For the remaining triples, we sum over $m_1$ first. For each $m_1$,
the proposition shows that the possible $(m_1/m_2,h/m_1)$ that can
occur lie, modulo $p$, in a finite set $\mathcal{E}$ of size bounded
in terms of the conductor of $\mcF$ only, i.e., $m_2$ modulo $p$ and
$h$ are determined by $m_1$ up to a finite number of possibilities. We
use the trivial bounds
$$
|\wwd(m_1,m_2,h,K)|\leq \cond(\sheaf{F})^2p,\quad\quad
\min\Bigl(\frac{1}{|h|},\frac{N}{p}\Bigr)\leq \frac{N}{p},
$$
and obtain that the contribution of these terms to the right-hand side
of \refs{eq-norm-bil} is
\begin{gather*}
  \ll \|\beta\|^2N\sum_{(t,h)\in\mathcal{E}} \sum_{m_1}
  \sum_\stacksum{m_1, m_2}{m_2\equiv tm_1\mods{p}}|\alpha_{m_1}||\alpha_{m_2}|\\
  \ll \|\beta\|^2N\sum_{(t,h)\in\mathcal{E}} \sum_{m_1}
  \sum_\stacksum{m_1, m_2}{m_2\equiv tm_1\mods{p}}(|\alpha_{m_1}|^2+|\alpha_{m_2}|^2)\\
  \ll N\Bigl(1+\frac{M}{p}\Bigr)\|\beta\|^2\|\alpha\|^2,
\end{gather*}
where the implicit constant depends only (polynomially) on
$\cond(\sheaf{F})$. 
\par
Combining the two, we get
$$
T\ll \|\alpha\|\|\beta\|
(MN)^{1/2}\Bigl(\frac{1}{p^{1/4}}+\frac{1}{M^{1/2}}+\frac{p^{1/4}\log^{1/2}p}
{N^{1/2}}\Bigr),
$$
where the implicit constant depends only on the conductor of $\sheaf{F}$. This completes the proof of Theorem \ref{typeIIsumthm}.

\section{Sums over primes}

We now finally prove Theorem \ref{primesumthm}, our main result on
sums over primes.

\subsection{Smooth sums}

We start with the smooth version \refs{primesumsmooth}. Clearly, it is
enough to estimate the sum
$$
\tsum_{V,X}(\Lambda,K)=\sum_{n}\Lambda(n)K(n)V\Bigl(\frac{n}{X}\Bigr),
$$
and we begin by recalling two lemmas. The first one is Heath-Brown's
identity for the von Mangoldt function~\cite{HB}:

\begin{lemma}[Heath-Brown]  
  For any integer  $J\geq 1$ and $n< 2X$, we have
$$
\Lambda(n)=-\sum_{j=1}^J(-1)^j\binom{J}{j} \sum_{m_1,\cdots, m_j\leq
  Z}\mu(m_1)\cdots\mu(m_j) \sum_{m_1\cdots m_jn_1\cdots
  n_j=n}\log n_1,
$$
where $Z=X^{1/J}$.
\end{lemma}

\begin{remark}
  Using instead the analogous formula
$$
\mu(n)=-\sum_{j=1}^J(-1)^j\binom{J}{j} \sum_{m_1,\cdots, m_j\leq
  Z}\mu(m_1)\cdots\mu(m_j) \sum_{m_1\cdots m_jn_1\cdots
  n_{j-1}=n}1,
$$
for the M\"obius function (valid under the same conditions), one
proves Theorem~\ref{moebiussumthm} using exactly the same arguments,
so we will not say more about the proof of that result.
\end{remark}

The second lemma provides a smooth partition of unity (see,
e.g.,~\cite[Lemma 2]{FouvryCrelle}).

\begin{lemma} \label{dyadic}
  There exists a sequence $(V_l)_{l\geq 0}$ of smooth functions on
  $[0,+\infty[$ such that
\begin{itemize}
\item[-] For any $l$, $V_{l}$ is supported in $]2^{l-1},2^{l+1}[$;
\item[-] For any $k,l\geq 0$, we have 
$$
x^kV_l^{(k)}(x)\ll_k 1,
$$
where the implicit constant depends only on $k$;
\item[-] For any $x\geq 1$,
$$
\sum_{l\geq 0}V_l(x)=1.
$$
\end{itemize}
\end{lemma}

Fix some $J\geq 2$. Applying these two lemmas, we see that
$\tsum_{V,X}(\Lambda,K)$ decomposes into a linear combination, with
coefficients bounded by $O_J(\log X)$, of $O(\log^{2J}X)$ sums of
the shape
\begin{multline}
\Sigma(\uple{M},\uple{N})=
\multsum_{m_1,\cdots,m_J}\alpha_{1}(m_1)\alpha_{2}(m_2)\cdots \alpha_{J}(m_J)\\
\times\multsum_{n_1,\cdots,n_J}V_{1}({n_1})\cdots V_{J}({n_J})
V\Bigl(\frac{m_1\cdots m_Jn_1\cdots n_J}{X}\Bigr)
K(m_1\cdots m_Jn_1\cdots n_J)
\end{multline}
where
\begin{itemize}
\item[-] $\uple{M}=(M_1,\cdots,M_J)$, $\uple{N}=(N_1,\cdots,N_J)$ are
  $J$-uples of parameters in $[1/2,2X]^{2J}$ which satisfy 
$$
N_1\geq N_2\geq \cdots \geq N_J,\quad\quad
M_i\leq X^{1/J},\quad\quad   M_1\cdots M_JN_1\cdots N_J\asymp_J X;
$$
\item[-] the arithmetic functions $m\mapsto \alpha_{i}(m)$ are bounded
  and supported in $[M_i/2,2M_i]$;
\item[-] the smooth functions $V_{i}(x)$ are compactly supported in
  $[N_i/2,2N_i]$, and their derivatives satisfy 
$$
y^{k}V_{i}^{(k)}(y)\ll 1,
$$
for all $y\geq 1$, where the implicit constants depend only on $k$.
\end{itemize} 

We will state different bounds for $\Sigma(\uple{M},\uple{N})$,
depending on the relative sizes of the parameters, and then optimize
the result.
\par
For $J\geq 2$, we obtain, by Theorem~\ref{typeIsumthm} applied
to $n_1,\ n_2$ and trivial summation over the remaining variables, the
bound
\begin{equation}
\label{eqbound2}
\Sigma(\uple{M},\uple{N})\ll (pQ)^\eps
QX\Bigl(1+\frac{p}{N_1N_2}\Bigr)^{1/2}p^{-1/8}.
\end{equation}
for any $\eps>0$, the implicit constant depending on $\eps$ and $\cond(\mcF)$.
\par
On the other hand, from \eqref{typeIIeq} with an integration by parts,
we have the bound
\begin{equation}
\label{eqbound3}
\Sigma(\uple{M},\uple{N})
\ll (pQ)^\eps
QX\Bigl(\frac{1}{p^{1/4}}+\frac1{M^{1/2}}+\frac{p^{1/4}}{(X/M)^{1/2}}
\Bigr),
\end{equation}
for any factorization 
$$
M_1\cdots M_JN_1\cdots N_J=M\times N
$$ 
where $M$ and $N$ are products of some of the $M_i$ and $N_{j}$. 
\par
Our goal is to choose the best of the two bounds~(\ref{eqbound2})
and~(\ref{eqbound3}) for each such configuration of the parameters
$(\uple{M},\uple{N})$. By taking logarithms (in base $p$), we readily
see that the proof of \refs{primesumsmooth} is reduced to the
optimization problem of the next section.

\subsection{An optimization problem}

We consider here the following optimization problem. We are given a
real number $x>0$ (we have in mind $x=\log X/\log p$), an integer
$J\geq 3$, and parameters
$$
(\bfm,\bfn)=(m_1,\cdots,m_J,n_1,\cdots,n_J)\in[0,x]^{2J}
$$
such that
\begin{equation}
\label{simplex}
\sum_{i}m_i+\sum_j n_j=x,\quad\quad
m_i\leq x/J,\quad\quad
n_1\geq n_2\geq \cdots\geq n_J.
\end{equation}
\par
We want to estimate from below the quantity
\begin{equation}\label{eq-first-max}
\eta(\bfm,\bfn)=\max\Bigl\{
\max_{\sigma}\min\Bigl(\frac{1}4,\frac{\sigma}2,\frac{x-\sigma
}2-\frac14\Bigr),\ \frac 18-\max\Bigl(0,\frac 12(1-(n_1+n_2))\Bigr)
\Bigr\},
\end{equation}
where $\sigma $ ranges over all possible sub-sums of the $m_i$ and
$n_j$ for $1\leq i,j\leq J$, that is over the sums
$$
\sigma=\sum_{i\in \mcI}m_i+\sum_{j\in \mcJ}n_j
$$
for $\mcI$, $\mcJ$ ranging over all possible subsets of
$\{1,\cdots,J\}$. 

\begin{remark}
  One could also try to exploit the estimate~(\ref{typeIeq}) to
  improve the result, but we will not actually use it.
\end{remark}

The number $\eta(\bfm,\bfn)$ represents the maximal power of $p$ that
we save over the trivial bound using \refs{eqbound2} and
\refs{eqbound3}.  The outcome of the discussion in the previous
section is that, for $x=(\log X)/(\log p)$ and $J\geq 3$, we have
$$
\Sigma_J(\uple{M},\uple{N})\ll
(pQ)^{\eps}QXp^{-\eta(\uple{m},\uple{n})}.
$$
\par
By Heath-Brown's identity, it follows that
$$
\tsum_{V,X}(\Lambda,K)\ll (pQ)^{\eps}QXp^{-\eta}
$$
where
$$
\eta=\min_{(\uple{m},\uple{n})}\eta(\uple{m},\uple{n}).
$$
\par
We will show:

\begin{proposition}\label{pr-optimize}
  Let $x>3/4$ be given. Provided $J$ is large enough in terms of $x$, we
  have the inequality
$$
\eta(\uple{m},\uple{n})\geq \min\Bigl(\frac{1}{24},\frac{4x-3}{24}\Bigr).
$$
\end{proposition}

Combining these lower-bounds with the above estimates, the proof of
Theorem~\ref{primesumthm} is concluded, noting that $x\leq 1$ means
that $X\leq p$, and that
$$
Xp^{-(4x-3)/24}=X\Bigl(\frac{p}{X}\Bigr)^{1/6}p^{-1/24}.
$$

\begin{proof}[Proof of Proposition~\ref{pr-optimize}]
Let $\delta$ be a parameter such that
\begin{equation}\label{eq-delta}
0<\delta < \min\Bigl(\frac{4x-3}{12},\frac{x-1/2}{6},\frac{1}{4}\Bigr).
\end{equation}
\par
The interval
$$
I_{\delta}=\Bigl[2\delta,x-\frac12-2\delta\Bigr]
$$
is then non-empty. If we can find a subsum $\sigma$ such that $\sigma
\in I_\delta$, we then deduce immediately from the
definition~(\ref{eq-first-max}) that
\begin{equation}\label{eq-first-bound}
\eta(\uple{m},\uple{n})\geq 
\max_\sigma\min\Bigl(\frac14,\frac{\sigma}2,\frac{x}2-\frac14-\frac \sigma
2\Bigr) \geq \delta.
\end{equation}
\par
We now assume that such a subsum $\sigma$ does \emph{not} exist, and
attempt to get a lower-bound on $\eta(\uple{m},\uple{n})$ using the
second term in the maximum~(\ref{eq-first-max}). First of all, we
claim that, in that case, we have
\begin{equation}
\label{xibound}
\sum_{i\leq J}m_i<2\delta,
\end{equation}
provided
\begin{equation}\label{eq-provided}
\frac{x}{J}\leq x-\frac12-4\delta=\mathrm{length}(I_\delta),
\end{equation}
a condition which we assume from now on. 
\par
Indeed, if~(\ref{xibound}) were false, using the fact that $m_i\leq
x/J$ and that $x/J$ is then less than the length of the interval
$I_\delta$, we would be able to find some subsum $\sigma$ (formed only
with some $m_i$'s) which is contained in $I_\delta$, contradicting our
current assumption.
\par
From \eqref{simplex} and  (\ref{xibound}), we get in particular  the inequality
\begin{equation}\label{eq-818}
\sum_{j}n_j\geq x-2\delta.
\end{equation}
\par
Since, under our assumption~(\ref{eq-delta}) on $\delta$, we have
$$
2\delta\leq x-\frac12-4\delta=\mathrm{length}(I_\delta),
$$ 
this implies that 
$$
n_j\leq 2\delta
$$
for any $j\geq 3$ (because otherwise, we would have
$$
x-\frac12-2\delta \leq n_3\leq n_2\leq n_1
$$ 
since $n_j\notin I_{\delta}$, and then, in view of~(\ref{eq-delta}),
we would get
$$
n_1+n_2+n_3> 3x-\frac32-6\delta\geq x,
$$
a contradiction). But now it follows that
\begin{equation}
\label{yjbound}
\sum_{j\geq 3}n_j<2\delta,
\end{equation}
because otherwise, using $4\delta\leq x-1/2-2\delta$, we could again
obtain a subsum of the $n_j$'s, $j\geq 3$, in $I_{\delta}$.
\par
Combining \refs{eq-818} and \refs{yjbound}, we obtain
$$
n_1+n_2\geq x-4\delta
$$
and hence
$$
\frac 18-\max\Bigl(0,\frac 12(1-(n_1+n_2))\Bigr)
\geq \min\Bigl(\frac18,\frac{4x-3}8-2\delta\Bigr).
$$
\par
Combining this with~(\ref{eq-first-bound}), it follows that for
$\delta$ satisfying~(\ref{eq-delta}) and $J$ large enough in terms of
$x$ and $\delta$ so that~(\ref{eq-provided}) holds, we have
$$
\eta(\bfm,\bfn)\geq\min\Bigl({\delta},\
\min\Bigl(\frac18,\frac{4x-3}8-2\delta\Bigr)\Bigr),
$$
\par
For $x>3/4$, we take
$$
\delta=\min\Bigl(\frac{4x-3}{24},\frac{1}{24}\Bigr)
$$
and Proposition \ref{pr-optimize} follows.
\end{proof}

\subsection{Sums over intervals}\label{ssec-intervals}

We can now also easily deduce from~(\ref{primesumsmooth}) the
estimate~(\ref{primesuminterval}) for sums over primes in the interval
$2\leq q\leq X$ (below all sums over $q$ are restricted to $q$ prime).
By a dyadic decomposition of the interval $[1,X]$, we are reduced to
proving that
\begin{equation}
\label{primesumdyadic}
\sum_{X\leq q\leq 2X}K(q)\ll_{\eta,\cond(\mcF)} X(1+p/X)^{1/12}p^{-\eta/2}
\end{equation}
for $X\geq 2$ and for any $\eta<1/24$. Since the right-hand side of
this bound increases with $X$, this is sufficient to conclude the
proof of \refs{primesuminterval}.

Let $\Delta<1$ be some parameter and let $V$ be a smooth function
defined on $[0,+\infty[$ such that
$$
\supp(V)\subset [1-\Delta,2+\Delta],\quad\quad
0\leq V\leq 1,\quad\quad V(x)=1\text{ for } 1\leq x\leq 2,
$$
and which satisfies
$$
x^{j}V^{(j)}(x)\ll_j Q^{j},
$$
with $Q=\Delta^{-1}$. 
\par
By applying \eqref{primesumsmooth} to $V$, we get
\begin{align*}
  \sum_{X\leq q\leq 2X}K(q)& \ll X\Delta +
  \sum_{q}K(q)V\Bigl(\frac{q}X\Bigr)\\
  &\ll_{\eta,\cond(\mcF)} X(\Delta +\Delta^{-1}(1+p/X)^{1/6}p^{-\eta})
\end{align*}
for any $\eta<1/24$.
\par
If $X>p^{1-6\eta}$, we can take
$$
\Delta=(1+p/X)^{1/12}p^{-\eta/2}<1
$$
and we obtain \refs{primesumdyadic}. On the other hand, if $X\leq
p^{1-6\eta}$, the bound \refs{primesumdyadic} is weaker than the
trivial bound $2X$ for $p$ large enough.

\section{Applications}

\subsection{Primes represented by a polynomial modulo $p$}
\label{sec:primepolynomials}

In this section we prove Corollaries~\ref{cor-poly-error-terms}
and~\ref{cor-mult-car}.
\par
For the former, we fix a non-constant polynomial $P\in\Zz[X]$, and we
consider a prime $p$ such that $P$ is non-constant modulo $p$.
\par
For Corollary~\ref{cor-poly-error-terms}, (1), we are dealing with
\begin{align*}
  \sum_{n\in\Fp}{E(X;p,P(n))}&= \sum_{n\in\Fp}{\pi(X;p,P(n))}
  -\frac{1}{p-1}\sum_\stacksum{n\in\Fp}{P(n)\not\equiv 0\mods{p}}{\pi(X)}.
\end{align*}
\par
We denote
$$
N_P(x)=\sum_{\stacksum{n\in\Fp}{P(n)=x}}{1}-1
$$
the ``centered'' number of representations of $x$ as a value of $P$
modulo $p$. The formula above allows us to write
$$
\sum_{n\in\Fp}{E(X;p,P(n))}
=\sum_{q\leq X}{N_P(q)}
+\sum_{q\leq X}\Bigl(1-\frac{1}{p-1}|\{n\in\Fp\,\mid\, P(n)\not=0\}|\Bigr)
$$
(where $q$ runs over primes, as before).
\par
The second term of the previous expression is trivially bounded by
$\ll p^{-1}X+1$, since $P$ has at most $\deg P$ zeros modulo $p$. Thus
Corollary~\ref{cor-poly-error-terms}, (1) follows from
Theorem~\ref{primesumthm} and from the fact -- recalled in
Section~\ref{subsec-decompositions-poly} below -- that $N_P$ is a
trace function for an $\ell$-adic sheaf with no exceptional
Jordan-H\"older factor (i.e. no such factor is geometrically
isomorphic to a tensor product of a Kummer sheaf and an Artin-Schreier
sheaf).
\par
For Corollary~\ref{cor-poly-error-terms}, (2), we write
$\mathbf{1}_{P(\Fp)}$ for the characteristic function of the set
$P(\Fp)$ of values of $P$ modulo $p$, and we will denote
$P^*(\Fp)=P(\Fp)-\{0\}$, the set of non-zero values of $P$ modulo
$p$. A reasoning similar to the previous one leads to
$$
\sum_{a\in P(\Fp)}{E(X;p,a)}= \sum_{q\leq X}\mathbf{1}_{P(\Fp)}(q)
-\frac{|P^*(\Fp)|}{p-1}\pi(X).
$$
\par
Applying Proposition~\ref{pr-decomp-poly} of
Section~\ref{subsec-decompositions-poly}, the first term on the
right-hand side becomes
$$
c_1\pi(X) +\sum_{2\leq i\leq k}\sum_{q\leq X}{c_i K_i(q)}+ O({p^{-1}}{X}+1)
$$
where the implicit constant depends only on $\deg P$, using the
notation of that proposition (the error term corresponds to the
contribution of those $q$ such that $q\mods{p}$ is in one of the
residue classes in the set $S$ of Proposition~\ref{pr-decomp-poly};
its size is bounded in terms of $\deg P$ only.)
\par
Using  the asymptotic formula~(\ref{eq-size-c1}) for the constant
$c_1$, we get
$$
\sum_{a\in P(\Fp)}{E(X;p,a)}=\sum_{2\leq i\leq k}\sum_{q\leq
  X}{c_iK_i(q)} +O(p^{-1/2}X),
$$
and Theorem~\ref{primesumthm} concludes the proof.

\subsection{Large Kloosterman sums with almost prime modulus}

In this section we prove Corollary \ref{largesums}. It is sufficient
to prove the following:

\begin{proposition}\label{largesumsII} 
  For any $m\geq 2$, and $\delta$ such that $0<\delta<1/2$, there
  exists a constant $\beta_m>0$ such that
$$
|\{(p,q),\ p,q\text{ primes }\geq X^{\delta},\ pq\leq X,\
|\hypk_m(1;pq)|\geq\beta_m\}|\gg \frac{X}{\log X}.
$$
here the implicit constants depend on $m$ and $\delta$ only.
\end{proposition}

We recall first the basic strategy from \cite{MichelInv}. By the
Chinese remainder theorem, we have the twisted multiplicativity
\begin{equation}
\label{twistedmultiplicativity}
\hypk_m(1;pq)=\hypk_m(\ov q^m;p)\hypk_m(\ov p^m;q),
\end{equation}
when $p$ and $q$ are distinct primes. Therefore, in order to prove the
existence of pairs of primes $(p,q)$ for which $|\hypk_m(1;pq)|$ is
large, it is sufficient to show that there exists two sets of pairs of
primes for which $|\hypk_m(\ov q^m;p)|$ and $|\hypk_m(\ov p^m;q)|$ are
both large, and that these two sets intersect non-trivially. This
leads us to proving that, for pairs $(p,q)$ in suitable ranges, the
hyper-Kloosterman sums $\hypk_m(\ov q^m;p)$ and $\hypk_m(\ov p^m;q)$
become equidistributed in the interval $[-m,m]$ with respect to a
suitable measure. Such a statement is an instance of the vertical (or
average) Sato-Tate laws of Katz and Deligne, but specialized to prime
arguments.
\par
To state properly these equidistribution statements, we recall that
for any prime number $p$ and auxiliary prime $\ell\not=p$, and for any
isomorphism $\iota:\ov{\Qq_\ell}\hookrightarrow \Cc$, there exists a
$\ov\Qq_{\ell}$-adic sheaf $\HYPK_m$ on $\Pp^1_\Fp$ (constructed by
Deligne and studied by Katz in~\cite{GKM}) such that:
\begin{enumerate}
\item The sheaf $\HYPK_m$ has rank $m$ and is lisse on ${\Gm}_{,\Fp}$,
  tamely ramified at $0$ with a single Jordan block and wildly
  ramified at $\infty$ with Swan conductor $1$ (in particular, we have
  $\cond(\HYPK)=m+3$);
\item The sheaf $\HYPK_m$ is geometrically irreducible, and its
  geometric monodromy group is equal to $\rmG_m=\SL_{m}$ or $\Sp_{m}$
  depending on whether $m$ is odd or even;
\item The sheaf $\HYPK_m$ is pointwise pure of weight $0$, and for any
  $a\in\Fpt$, the trace of the Frobenius at $a$ equals
$$
\iota(\tr(\frob_a|\HYPK_m))=(-1)^{m-1}\hypk_m(a;p),
$$
and moreover, for any choice of maximal compact subgroup $K_m$ of
$\rmG_m(\Cc)$, $(\frob_a|\HYPK_m)$ defines a unique conjugacy class
$\gnat_m(a;p)$ in $K_m^\natural$ (the space of conjugacy classes of $K_m$)  whose trace is equal to $(-1)^{m-1}\hypk_m(a;p)$.
\end{enumerate}

It will be easy to prove the following result using Theorem
\ref{primesumthm}:

\begin{theorem}[Sato-Tate equidistribution]\label{equidthm} 
  Given $\delta<0$, $A\geq 1$, $P,Q\geq 2$ such that
$$
P^{3/4+\delta}\leq Q\leq P^A, 
$$
the set of conjugacy classes
$$
\{\gnat_m({\ov q^m;p}),\ p\not=q\text{ primes, }
(p,q)\in[P,2P]\times[Q,2Q]\}\subset K_m^\natural,
$$
becomes equidistributed as $P\ra +\infty$ with respect to the (image of the)
probability Haar measure $\mu_{m}$ on $K_m^\natural$.
\end{theorem}

\begin{remark}
  A similar Sato-Tate equidistribution result over the primes holds
  for the generalized Kloosterman sheaves of Heinloth, Ng\^o and
  Yun~\cite{HNY} already mentioned in Remark \ref{introHNY}. 
\end{remark}
  
We will sketch the proof below, but for the moment we can conclude
from this the proof of Corollary~\ref{largesums}. We pick $\alpha_m>0$
small enough such that
$$
\mu_{m}(\{\gnat\in K_m^\natural,\ |\tr(\gnat)|\geq\alpha_m\})\geq
0.51,
$$ 
(such an $\alpha_m$ exists because the direct image of the measure
$\mu_{m}$ under the trace map $\gnat\mapsto |\tr(\gnat)|$ is
absolutely continuous with respect to the Lebesgue probability measure
on $[0,m]$). 
\par
Now let $\delta>0$ be given, let $P$ be large enough and consider $Q$
such that 
$$
P^{3/4+\delta}\leq Q\leq P^{4/3-\delta}.
$$
\par
We then have 
$$
Q^{3/4+\delta'}\leq P\leq Q^{4/3-\delta'}
$$
for some $\delta'>0$ depending only on $\delta$, and we can apply
Theorem~\ref{equidthm} twice to show that \emph{both} sets
\begin{gather*}
  \mathcal{P}_1=\{(p,q)\in[P,2P]\times[Q,2Q],\ p\not=q\text{ primes, }
  |\tr(\gnat_m({\ov q^m;p}))|\geq \alpha_m\}\\
  \mathcal{P}_2= \{(p,q)\in[P,2P]\times[Q,2Q],\ p\not=q\text{ primes,
  } |\tr(\gnat_m({\ov p^m;q}))|\geq \alpha_m\}
\end{gather*}
satisfy, as $P\ra +\infty$, the limit
$$
\frac{|\mathcal{P}_i|}{|\{(p,q)\in [P,2P]\times[Q,2Q]\}|}\lra 
\mu_{m}(\{\gnat\in K_m^\natural,\ |\tr(\gnat)|\geq\alpha_m\})\geq
0.51.
$$
\par
In particular, the two sets have a non-empty intersection for $P$
large enough, and in fact
$$
|\mathcal{P}_1\cap \mathcal{P}_2|\gg \frac{P}{\log P}\frac{Q}{\log Q}.
$$
\par
By~\refs{twistedmultiplicativity}, it follows that
$$
|\{(p,q)\in[P,2P]\times[Q,2Q],\ p\not=q\text{ primes, }
|\hypk_m(1;pq)|\geq \alpha^2_m\}|\gg \frac{P}{\log P}\frac{Q}{\log Q}.
$$
\par
Then we obtain by an easy argument of dyadic partition that for $X$
large enough, we have
$$
|\{(p,q),\ p\not=q\text{ primes, } \ p,q\geq X^{4/9},\ pq\in[X,2X],\
|\hypk_m(1;pq)|\geq \alpha^2_m\}|\gg \frac{X}{\log X},
$$
as claimed.

\begin{proof}[Proof of Theorem \ref{equidthm}]
This is a direct application of the Weyl criterion. Let
$$
X_{P,Q}=\{ p\not=q\text{ primes, } (p,q)\in[P,2P]\times[Q,2Q]\}.
$$
\par
It is enough to prove that if $\rho$ is a non-trivial irreducible
representation of $\rmG_m$, we have
\begin{equation}\label{eq-weyl}
\frac{1}{|X_{P,Q}|}
\sum_{(p,q)\in X_{P,Q}}
\tr\rho(\gnat_m({\ov q^m;p}))
\lra 0
\end{equation}
as $P\ra +\infty$. 
\par
Now, for each $p$, we can interpret the sum over $q$ as the sum of the
weight 
$$
K_{\rho}(q)=\tr\rho(\gnat_m({\ov q^m;p}))
$$
modulo $p$. Now we claim that, for each $\rho\not=1$, the weight
$K_{\rho}$ is a non-exceptional irreducible trace weight modulo $p$
with conductor bounded by a constant depending only on $m$ and
$\rho$. Assuming this, Theorem~\ref{primesumthm}
(see~\eqref{primesuminterval})
gives
$$
\sum_{(p,q)\in X_{P,Q}} \tr\rho(\gnat_m({\ov q^m;p})) \ll
\frac{PQ}{\log P}P^{-\eta}\Bigl(1+\frac{P}{Q}\Bigr)^{1/12}
$$
for any $\eta<1/48$. Dividing by $|X_{P,Q}|\asymp PQ/(\log P)(\log
Q)$, we get
$$
\frac{1}{|X_{P,Q}|}
\sum_{(p,q)\in X_{P,Q}}
\tr\rho(\gnat_m({\ov q^m;p}))\ll (\log Q)(1+P/Q)^{1/2}P^{-\eta},
$$
which tends to $0$ provided $P^{3/4+\delta}<Q<P^A$ for some
$\delta>0$, $A\geq 1$.
\par
To check the claim, we first define
$$
\HYPK_m'=[x\mapsto x^{-m}]^*\HYPK
$$
so that, for $a\in\Fpt$, we have the trace function
$$
\iota(\frtr{\HYPK'_m}{\Fp}{a})=(-1)^{m-1}\hypk_m(a^{-m};p).
$$
\par
The function
$$
K_{\rho}\,:\, a\mapsto \tr(\rho(\gnat_m(a^{-m};p)))
$$
is then (the restriction to $\Fpt$ of) the irreducible trace weight
associated to the sheaf $\rho(\HYPK'_m)$ obtained by composing the
representation $\HYPK'_m$ with the representation $\rho$. In
particular, this sheaf is also lisse and geometrically irreducible on
$\Gg_m$, and of rank $\dim\rho$. It is tame at $\infty$, and its Swan
conductor at $0$ is bounded in terms of $m$ and $\dim\rho$ only (by
bounding the largest slope, see e.g.~\cite{MichelInv}), so the
conductor is bounded in terms of $m$ and $\deg \rho$ only. Finally,
because $\rho(\HYPK'_m)$ is irreducible of rank $\deg \rho\geq 2$ (we
use here the fact that both $\SL_m$ and $\Sp_m$ have no non-trivial
representations of dimension $1$), it follows that $\rho(\HYPK'_m)$ is
not $p$-exceptional.
\end{proof}

\section{Results from algebraic geometry}\label{sec-prelim}

\subsection{Properties of the Fourier-M\"obius group}

The goal of this section is to prove
Proposition~\ref{correlationprop}. In order to do so, we must first
recall the definition of the Fourier-M\"obius group of an isotypic
sheaf $\sheaf{F}$, and establish a few of its properties which were
not necessary in~\cite{FKM}.
\par
Let $p$ be a prime. Let $\ell\not=p$ be an auxiliary prime number,
$\iota\,:\, \bar{\Qq}_{\ell}\simeq \Cc$ an isomorphism. Let $\psi$ be
the $\ell$-adic additive character such that $\iota(\psi(x))=e(x/p)$
for $x\in\Fp$. 
\par
Given any middle-extension sheaf $\sheaf{F}$ on $\Aa^1_{\Fp}$, any
finite extension $k/\Fp$ and any $x\in \Pp^1(k)$, we denote by
$$
\frtr{\sheaf{F}}{k}{x}
$$
the trace of the geometric Frobenius of $k$ acting on the stalk of
$\sheaf{F}$ at $x$. We also denote by $\dual(\sheaf{F})$ the
middle-extension dual of $\sheaf{F}$ given by
$j_*(\check{j^*\sheaf{F}})$, where $j\,:\, U\injecte \Pp^1$ is the
inclusion of any dense open set $U$ on which $\sheaf{F}$ is lisse. 
\par
If $\sheaf{F}$ is any Fourier sheaf (in the sense
of~\cite[Def. 8.2.2]{GKM}) on $\Aa^1_{\Fp}$, we denote by
$\ft(\sheaf{F})$ the Fourier transform of $\sheaf{F}$, computed by
means of $\psi$, which satisfies
$$
\frtr{\ft(\sheaf{F})}{\Fp}{y}=
-\sum_{x\in \Fp}{\frtr{\sheaf{F}}{\Fp}{x}\psi(xy)}
$$
for any $y\in\Fp$. It follows from \cite[8.4.1]{GKM}, that
$\ft(\sheaf{F})$ is geometrically isotypic (resp. geometrically
irreducible) if $\sheaf{F}$ is isotypic (resp. geometrically
irreducible.) 
\par
Let now $\sheaf{F}$ be an isotypic trace sheaf modulo $p$ as in
Definition~\ref{def-admissible}. In~\cite{FKM}, we defined the
Fourier-M\"obius group of $\sheaf{F}$ by
$$
\haut_{\sheaf{F}}=\{\gamma\in\PGL_2(\bar{\Ff}_p)\,\mid\, 
\gamma^*(\ft(\sheaf{F}))\simeq \ft(\sheaf{F})\},
$$
where $\simeq$ denotes geometric isomorphism
(see~\cite[Def. 1.14]{FKM}).  Furthermore, we defined the correlation
sums of $\sheaf{F}$ by
$$
\wwd(\sheaf{F};\gamma)= \frac{1}{p}
\sum_{x\in\Fp}{\frtr{\ft(\sheaf{F})}{\Fp}{\gamma\cdot x}
  \overline{\frtr{\ft(\sheaf{F})}{\Fp}{x}}}
$$
for $\gamma\in\PGL_2(\Fp)$.
\par
The crucial link between these two notions is the following result
(see~\cite[Cor. 9.2]{FKM}) which follows from the Riemann Hypothesis
over finite fields, and from bounds for the conductor of the Fourier
transforms of Fourier sheaves.

\begin{proposition}
  Let $p$ be a prime and let $\mcF$ be an isotypic trace sheaf modulo
  $p$. There exists $M\geq 1$, which depends only, polynomially, on
  $\cond(\sheaf{F})$, such that
$$
|\iota(\wwd(\sheaf{F};\gamma))|\leq M\sqrt{p}
$$
for all $\gamma\notin \haut_{\sheaf{F}}$. 
\end{proposition}

Let then
$$
\hautb_{\sheaf{F}}=\haut_{\sheaf{F}}\cap \mathbf{B},
$$
where $\Bb\subset \PGL_2$ is the upper-triangular Borel subgroup. We
deduce from the proposition above:

\begin{proposition}
  Let $p$ be a prime, let $\sheaf{F}$ be an isotypic trace sheaf
  modulo $p$, and let $K(x)=\iota(\frtr{\sheaf{F}}{\Fp}{x})$ denote
  the trace function of $\sheaf{F}$ on $\Fp$.  There exists $M\geq 1$,
  depending only, polynomially, on $\cond(\sheaf{F})$, such that
$$
\Bigl|\sum_{x\in\Fp}{K(x)\overline{K(ax)}e\Bigl(\frac{bx}{p}\Bigr)}
\Bigr|\leq M\sqrt{p}
$$
if
\begin{equation}\label{eq-borel}
\begin{pmatrix}
  a&b\\0&1
\end{pmatrix}\notin \hautb_{\sheaf{F}}.
\end{equation}
\end{proposition}

\begin{proof}
By means of the Plancherel formula for the finite-field Fourier
transform, we check easily that
$$
\sum_{x\in\Fp}{K(x)\overline{K(ax)}e\Bigl(\frac{bx}{p}\Bigr)}
=\iota(\wwd(\sheaf{F};\gamma))
$$
where $\gamma$ is the upper-triangular matrix
in~(\ref{eq-borel}). Hence the proposition gives the result.
\end{proof}

It follows now that Proposition~\ref{correlationprop} is a consequence
of the next theorem:

\begin{theorem}\label{th-bound-bad}
  Let $p$ be a prime and let $\sheaf{F}$ be an isotypic sheaf. At least one of the following four properties holds:
\par
\emph{(1)} The trace function of $\sheaf{F}$ is proportional to a
delta function at some point $a\in\Fp$, or to the trace function of a
sheaf $\sheaf{L}_{\psi(aX)}$ for some $a\in\Fp$, i.e., to an additive
character;
\par
\emph{(2)} The group $\hautb_{\sheaf{F}}$ has dimension $\geq 1$ and
$\sheaf{F}$ is $p$-exceptional, i.e., its unique geometrically
irreducible component is a tensor product
$\sheaf{L}_{\chi}\otimes\sheaf{L}_{\eta}$ for some non-trivial Kummer
sheaf $\sheaf{L}_{\chi}$ and some possibly trivial additive character
$\eta$;
\par
\emph{(3)} The group $\hautb_{\sheaf{F}}$ is finite and
$$
|\hautb_{\sheaf{F}}(\Fp)|\leq 10\cond(\sheaf{F})^2;
$$
\par
\emph{(4)} The conductor of $\sheaf{F}$ is at least $(p/10)^{1/2}$.
\end{theorem}

To prove this, we first prove two basic properties of the
Fourier-M\"obius group and one lemma concerning Swan conductors.

\begin{proposition}\label{pr-algebraic}
  Let $k$ be a finite field, and let $\sheaf{F}$ be an $\ell$-adic
  isotypic trace sheaf on $\Aa^1_k$. Let $\sheaf{G}$ be its Fourier
  transform. Then the subgroup $\haut_{\sheaf{F}}\subset
  \PGL_2(\bar{k})$ is an algebraic subgroup defined over $k$.
\par
In particular, for $\sheaf{F}$ over $\Fp$, $\hautb_{\sheaf{F}}$ is an
algebraic subgroup of $\Bb$ defined over $\Fp$.
\end{proposition}

We thank R. Pink for explaining to us how to prove this proposition.

\begin{proof}
  Let $S\subset \Pp^1$ be the divisor of singularities of $\sheaf{G}$,
  so that $U=\Pp^1-S$ is the largest open set on which it is
  lisse. Because $\sheaf{G}$ is non-constant (the sheaf $\sheaf{F}$
  would have to be a Dirac delta sheaf supported on a single point for
  this to happen, and such a sheaf is not a Fourier sheaf), we have
  $S\not=\emptyset$. Let $G\subset \PGL_2$ be the stabilizer of $S$,
  which is a proper algebraic subgroup of $\PGL_2$ defined over
  $\Fp$. Then we have a first inclusion $\haut_{\sheaf{F}}\subset G$.
\par
Now we work over $\bar{k}$, and just denote by $U$ its base-change to
$\bar{k}$. We consider the action morphism
$$
\mu\,:\,\begin{cases}
G\times U\lra U\\
(\gamma,x)\mapsto \gamma\cdot x
\end{cases}
$$
and the second projection $p_2\,:\, G\times U\lra U$, and we define
the sheaf
$$
\sheaf{E}=\mu^*\sheaf{G}\otimes p_2^*\dual(\sheaf{G})
$$
on $G\times U$ and the higher direct-image
$\sheaf{I}=R^2p_{1,!}\sheaf{E}$, which is a sheaf on the algebraic
group $G/\bar{k}$. By the base-change theorem for higher-direct images
with compact support~\cite[Arcata, IV, Th. 5.4]{deligne}, the stalk of
$\sheaf{I}$ at a geometric point $\gamma\in G(\bar{k})$ is naturally
isomorphic to $H^2_c(U,\gamma^*\sheaf{G}\otimes\dual(\sheaf{G}))$.
\par
Furthermore, the constructibility theorem for higher direct images
with compact support~\cite[Arcata, IV, Th. 6.2]{deligne} shows that
$\sheaf{I}$ is a constructible $\ell$-adic sheaf on $G$. This implies
(see also~\cite[Rapport, Prop. 2.5]{deligne}) that for any $d\geq 0$,
the set
$$
\{\gamma\in G(\bar{k})\,\mid\, \dim \sheaf{I}_{\gamma}=\dim H^2_c(U,
\gamma^*\sheaf{G}\otimes\dual(\sheaf{G}))=d\}
$$
is constructible in $G(\bar{k})$, i.e., is a finite union of
locally-closed subsets. In particular, the set of all $\gamma$ where
$$
H^2_c(U,\gamma^*\sheaf{G}\otimes\dual(\sheaf{G}))\not=0,
$$
is constructible. But this set is exactly $\haut_{\sheaf{F}}$ by the
co-invariant formula for $H^2_c$ on a curve
(see~\cite[Th. 9.1]{FKM}). Since it is well-known that a constructible
subgroup of an algebraic group is Zariski-closed (see,
e.g.,~\cite[Ch. I, Prop. 1.3]{borel}) we conclude therefore that
$\haut_{\sheaf{F}}$ is a closed subgroup of $\PGL_2$.
\par
Finally, $\haut_{\sheaf{F}}$ is defined over $k$: since $\sheaf{F}$ is
invariant under the Frobenius automorphism of $k$, the definition
implies that if $\gamma\in \haut_{\sheaf{F}}$, then so does the image
of $\gamma$ under the Frobenius automorphism.
\end{proof}


Next we need to understand when $\haut_{\sheaf{F}}$ can be
``large''. We prove here a bit more than what we need for the sake of
completeness. We use the notation $\rmT^{x,y}$ for the maximal torus
in $\PGL_2$ defined as the pointwise stabilizer of $\{x,y\}\subset
\Pp^1$ (for $x\not=y$) and $\rmU^x$ for the unipotent radical of the
Borel subgroup $\Bb^x$ which is the stabilizer of $x\in \Pp^1$.

\begin{proposition}\label{pr-big-fourier}
  Let $\sheaf{F}$ be a geometrically isotypic $\ell$-adic Fourier
  sheaf on $\Aa^1_{\Ff_p}$, with Fourier transform
  $\sheaf{G}=\ft_{\psi}(\sheaf{F})$ with respect to some non-trivial
  additive character $\psi$.
\par
\emph{(1)} If there exists $x\in \Pp^1$ such that
$\haut_{\sheaf{F}}\supset \rmU^x$, then $\sheaf{G}$ is geometrically
isomorphic to a direct sum of copies of
$\sheaf{L}_{\psi_0(\gamma_0(X))}$ for some non-trivial additive
character $\psi_0$, where $\gamma_0\in \PGL_2$ is such that
$\gamma_0\cdot x=\infty$. In that case, we have
$\haut_{\sheaf{F}}=\rmU^x$.
\par
\emph{(2)} If there exist $x\not=y$ in $\Pp^1$ such that
$\haut_{\sheaf{F}}\supset \rmT^{x,y}$, then $\sheaf{G}$ is
geometrically isomorphic to a direct sum of copies of
$\sheaf{L}_{\chi_0(\gamma_0(X))}$ for some non-trivial multiplicative
character $\chi_0$, where $\gamma_0\in \PGL_2$ is such that
$\gamma_0\cdot x=0$, $\gamma_0\cdot y=\infty$. In that case, we have
$\haut_{\sheaf{F}}=\rmT^{x,y}$ if $\chi_0$ is not of order $2$, and
$\haut_{\sheaf{F}}=\rmN^{x,y}$, the normalizer of $\rmT^{x,y}$, if
$\chi_0^2=1$.
\end{proposition}

\begin{proof}
  (1) The ``if'' direction is immediate. For the converse, we may
  first assume that $x=\infty$, by conjugation with a matrix
  $\gamma_0$ with $\gamma_0\cdot x=\infty$. The assumption is then
  that
$$
\begin{pmatrix}1&t\\
0&1\end{pmatrix}^*\sheaf{G}\simeq \sheaf{G},
$$
for any $t\in\bar{\Ff}_p$, where the symbol $\simeq$ denotes geometric
isomorphism. Since $\sheaf{G}$ is geometrically isotypic, we also have
$$
\begin{pmatrix}1&t\\
0&1\end{pmatrix}^*\sheaf{G}_1\simeq \sheaf{G}_1,
$$
for $t\in \bar{\Ff}_p$, where $\sheaf{G}_1$ is the geometrically
irreducible component of $\sheaf{G}$.  We can then apply~\cite[Lemma
2.6.13]{katz-rls} to deduce that
$$
\sheaf{G}_1\simeq\sheaf{L}_{\psi_0(X)}
$$
(geometrically) for some additive $\ell$-adic character $\psi_0$, and
hence $\sheaf{G}$ is a direct sum of copies of this Artin-Schreier
sheaf. Furthermore, it follows from the classification of
Artin-Schreier sheaves that if $\psi_0$ is non-trivial and
$\gamma\notin \rmU^{\infty}$, we do not have
$\gamma^*\sheaf{L}_{\psi_0(X)}\simeq \sheaf{L}_{\psi_0(X)}$, and
therefore the Fourier-M\"obius group is exactly equal to
$\rmU^{\infty}$.
\par
(2) As before, we may first conjugate using some $\gamma_0$ to reduce
to the case where $x=0$, $y=\infty$, and we may reduce to the case
where $\sheaf{F}$ and $\sheaf{G}$ are geometrically irreducible, so
that the assumption is
$$
\haut_{\sheaf{F}}\supset \rmT=\rmT^{0,\infty}=\Bigl\{\begin{pmatrix}a&0\\0&d
\end{pmatrix}\Bigr\}
$$
for all $a$, $d\in\bar{k}$. By~\cite[Lemma 2.6.13]{katz-rls}, again,
there exists a multiplicative character $\chi_0$ such that
$$
\sheaf{G}\simeq \sheaf{L}_{\chi_0(X)}.
$$
\par
This character is non-trivial since $\sheaf{G}$ is a Fourier
sheaf. Now to finish the computation of $\haut_{\sheaf{F}}$, we use
the fact that $\sheaf{L}_{\chi_0(X)}$ is tamely ramified at $0$ and
$\infty$, and hence
$$
\haut_{\sheaf{F}}\subset \rmN=\rmN^{0,\infty}=\rmT\cup \Bigl\{
\begin{pmatrix}
0&b\\c&0
\end{pmatrix}
\Bigr\},
$$
the normalizer of $\rmT$ in $\PGL_2$. Clearly, $\rmT\subset
\haut_{\sheaf{L}_{\chi_0(X)}}$. If $\gamma\in \rmN-\rmT$, on the other
hand, we have $\gamma^*\sheaf{F}_{\chi_0(X)}\simeq
\sheaf{F}_{\chi_0(X^{-1})}$, and by the classification of Kummer
sheaves, it follows that $\gamma\in \haut_{\sheaf{L}_{\chi_0(X)}}$ if
and only if $\chi_0=\chi_0^{-1}$, i.e., if $\chi_0$ is of order $2$.
\end{proof}

The second lemma concerns the size of Swan conductors of lisse sheaves
on $\Gg_m$ with some non-trivial (multiplicative)
translation-invariance property.

\begin{lemma}\label{lm-mult-invariant}
  Let $k$ be an algebraic closure of a finite field of characteristic
  $p$, and let $\sheaf{F}$ be an $\ell$-adic sheaf for some
  $\ell\not=p$ which is lisse on $\Gg_{m,k}$.  If there exists
  $a\not=1$ in $\Gg_m(k)$ such that $\sheaf{F}\simeq [\times
  a]^*\sheaf{F}$, then $m\mid \swan_{\infty}(\sheaf{F})$, where $m$ is
  the multiplicative order of $a$. In particular, if $\sheaf{F}$ is
  not tame at $\infty$, we have $\swan_{\infty}(\sheaf{F})\geq m$.
\end{lemma}

\begin{proof}
  Let $V$ be the generic stalk of $\sheaf{F}$, seen as a
  representation of the inertia group $I=I(\infty)$ at $\infty$, and
  let
$$
V=\bigoplus_{\alpha\in A}{V_{\alpha}}
$$
be the decomposition of $V$ in $I$-isotypic subspaces. Let $W_{\alpha}$
denote the irreducible $I$-representation such that $V_{\alpha}$ is a
multiple of $W_{\alpha}$.
\par
The finite cyclic subgroup $G\subset \Gg_m(k)$ of order $m$ generated
by $a$ acts on the index set $A$, corresponding to the fact that
$[\times a]^*V=V$ as $I$-representation: we have
$$
[\times a^j]^*V_{\alpha}=V_{a^j\cdot \alpha},
$$
for any integer $j\geq 0$, and in fact even
$$
[\times a^j]^*W_{\alpha}=W_{a^j\cdot \alpha},
$$
since $W_{\alpha}$ is uniquely determined by $V_{\alpha}$. 
\par
Let $B\subset A$ be one of the orbits of $G$. Its size $|B|$ is a
divisor of $m$, and if $\alpha\in B$, we have an isomorphism $[\times
a^{|B|}]^*W_{\alpha}=W_{\alpha}$. Since $W_{\alpha}$ is irreducible, we can
apply~\cite[Prop. 4.1.6 (2)]{GKM} to deduce that
$$
\swan_{\infty}(W_{\alpha})\equiv 0\mods{m/|B|}.
$$
\par
Since multiplicative translation by $a$ is an automorphism, it follows
that
$$
\swan_{\infty}(W_{a^j\cdot \alpha})=\swan_{\infty}(W_{\alpha})\equiv
0\mods{m/|B|}
$$
for any $a^j\in G$. Summing over the orbit, we get
$$
\swan_{\infty}\Bigl(\bigoplus_{\alpha\in B}{V_{\alpha}}\Bigr)\equiv 0\mods{m},
$$
and then summing over the orbits we get
$$
\swan_{\infty}(V)\equiv 0\mods{m}.
$$
\par
If $\sheaf{F}$ is wild at infinity, than $\swan_{\infty}(V)\not=0$,
and therefore it must be $\geq m$.
\end{proof}

Having dealt with these preliminaries, we can now prove the theorem.

\begin{proof}[Proof of Theorem~\ref{th-bound-bad}]
  The group $B=\hautb_{\sheaf{F}}(\Fp)$ is a finite subgroup of
  $\Bb\cap \PGL_2(\Fp)$. We distinguish three situations in turn.
\par
(1) If $B$ contains a non-trivial unipotent element $g$, then since
$g$ fixes $\infty$, the reasoning in~\cite[\S 9, Proof of
Th. 1.12]{FKM} shows that either $\cond(\sheaf{G})\geq p$, in which
case the fourth case holds by~\cite[Prop. 8.2 (1)]{FKM}, or otherwise
the trace function of the Fourier transform $\ft_{\psi}(\sheaf{F})$ is
proportional to an additive character, so that the trace function of
$\sheaf{F}$ is proportional to a delta function, and we are in the
first case.
\par
Now, if $B$ contains no unipotent elements, the unipotent radical of
$\hautb_{\sheaf{F}}$ must also be trivial (otherwise it would have
non-trivial $\Fp$-points). So, by the structure of $\Bb$, the
connected component of the identity $\hautb_{\sheaf{F}}^{\circ}$ of
$\hautb_{\sheaf{F}}$ is contained in a conjugate (say $\Dd$) of the
diagonal subgroup in $\Bb$. Since $\Dd$ has dimension $1$, there are
two further possibilities:
\par
(2) If $\hautb_{\sheaf{F}}^{\circ}=\Dd$, so that
$\hautb_{\sheaf{F}}\supset \Dd$, we deduce from
Proposition~\ref{pr-big-fourier} (2) that the Fourier transform of
$\sheaf{F}$ is geometrically isomorphic to a direct sum of copies of
$\sheaf{L}_{\chi(\gamma(X))}$ for some multiplicative character $\chi$
and some $\gamma\in \Bb$. By Fourier transform, this implies that
$\sheaf{F}$ is geometrically isomorphic to a direct sum of copies of
the tensor product $\sheaf{L}_{\chi}\otimes\sheaf{L}_{\eta}$ for some
multiplicative character $\chi$ and some additive character
$\eta$. Here $\chi$ must be non-trivial because otherwise $\sheaf{F}$
would not be a Fourier sheaf, and we are in the second case of the
statement of the proposition.
\par
(3) Otherwise, $\hautb_{\sheaf{F}}$ is a finite group so that its
finite subgroup $B\subset \Dd$ is cyclic, and there exists
$x_0\in\Aa^1$ such that all elements of $B$ fix $\infty$ and
$x_0$. Let $\sheaf{G}$ be the Fourier transform of
$\sheaf{F}$. Replacing $\sheaf{G}$ with
$\sheaf{G}_0=[-x_0]^*\sheaf{G}$, which has the same conductor as
$\sheaf{G}$, we can assume that $x_0=0$, and hence that $B$ can be
identified with a finite cyclic subgroup of $\Fpt$ acting on $\Pp^1$
by multiplication.  Let $a\in \Fpt$ be a generator of $B\subset
\Fpt$. There are two subcases:
\par
-- (3.1) If $\sheaf{G}_0$ is not lisse on $\Gg_m$, there is a non-zero
singularity $s\in\Gg_m$ of $\sheaf{G}_0$; the geometric isomorphism
$\sheaf{G}_0\simeq [\times a]^*\sheaf{G}_0$ implies that the orbit of
$s$ under multiplication by powers of $a$ is also contained in the set
$S$ of singularities of $\sheaf{G}_0$. This set contains $\geq |B|$
elements, and therefore
$$
\cond(\sheaf{G})=\cond(\sheaf{G}_0)\geq |S|\geq |B|
$$
in that case, and by~\cite[Prop. 8.2 (1)]{FKM}, we get
$$
|B|\leq \cond(\sheaf{G})\leq 10\cond(\sheaf{F})^2,
$$
i.e., case (3) of the theorem.
\par
-- (3.2) If $\sheaf{G}_0$ is lisse on $\Gg_m$, we first note that
$\sheaf{G}_0$ cannot be tame at both $0$ and $\infty$, since the tame
fundamental group of $\Gg_m$ is abelian and $\sheaf{G}_0$ would then
be a Kummer sheaf, which we excluded by assuming that
$\hautb_{\sheaf{F}}$ is finite (again from
Proposition~\ref{pr-big-fourier}, (2)). Up to applying a further
automorphism $x\mapsto x^{-1}$, we can assume that $\sheaf{G}_0$ is
wildly ramified at $\infty$. We can then apply
Lemma~\ref{lm-mult-invariant} to $\sheaf{G}_0$, and deduce that
$$
\swan_{\infty}(\sheaf{G}_0)\geq |B|,
$$
and hence we get again
$$
\cond(\sheaf{G})=\cond(\sheaf{G}_0)\geq
\swan_{\infty}(\sheaf{G}_0)\geq |B|,
$$
and conclude as before.
\end{proof}

\subsection{Decomposition of characteristic functions}
\label{subsec-decompositions-poly}

In this section, we explain the necessary properties of the trace
weights underlying Corollary~\ref{cor-poly-error-terms}. We recall
especially the decomposition of the characteristic function of the set
of values $P(n)$ of a polynomial $P\in\Fp[X]$ with $n\in\Fp$ in terms
of trace functions. These types of results are well-known, but we give
the full proof since we require some quantitative information
concerning this decomposition.

\begin{proposition}\label{pr-decomp-poly}
  Let $p$ be prime and let $P\in\Fp[X]$ be a non-constant polynomial
  of degree $\deg P<p$. Let $\mathcal{P}$ be the set of values of $P$
  modulo $p$ and let $\mathbf{1}_P$ be its characteristic function.
\par
There exist a finite set $S\subset \Fp$ with order at most $\deg P$,
an integer $k\geq 1$ and a finite number of trace functions $K_i$
associated to middle-extension sheaves $\sheaf{F}_i$, $1\leq i\leq k$,
which are pointwise pure of weight $0$, and algebraic numbers
$c_i\in\bar{\Qq}$, such that
\begin{equation}\label{eq-decomp-charfun}
\sum_{i}c_iK_i(x)=\mathbf{1}_{P}(x)
\end{equation}
for all $x\in \Fp-S$, and with the following properties:
\par
-- The constants $k$, $|c_i|$ and $\cond(\sheaf{F}_i)$ are bounded
in terms of $\deg P$ only;
\par
-- The sheaf $\sheaf{F}_1$ is trivial and none of the $\sheaf{F}_i$
for $i\not=1$ are  geometrically trivial, and furthermore
\begin{equation}\label{eq-size-c1}
c_1=\frac{|\mathcal{P}|}{p}+O(p^{-1/2}),
\end{equation}
where the implicit constant depends only on $\deg P$;
\par
-- If $P$ is squarefree, no $\sheaf{F}_i$, $i\not=1$, contains an
exceptional sheaf as a Jordan-Hölder factor. 
\end{proposition}

\begin{proof}
Let $K(x)$, for $x\in \Fp$, denote the characteristic function of the
set of values $P(y)$ for $y\in \Fp$, so that we are trying to express
$K$ as a linear combination of trace weights.
\par
Let $\tilde{D}\subset \Aa^1$ be the critical points of $P$,
$\tilde{S}=P(\tilde{D})\subset \Aa^1$ the set of critical values, so
that $P$ restricts to a finite \'etale covering
$$
V=\Aa^1-\tilde{D}\lra U=\Aa^1-\tilde{S}
$$
and let
$$
W\fleche{\pi} V\lra U
$$
be the Galois closure of $V$. The Galois group $G=\Gal(W/U)$ contains
the subgroup $H=\Gal(W/V)$, and has order dividing $\deg(P)!$, hence
coprime to $p$.
\par
For any $x\in U(\Fp)$, the Galois group $G$ permutes the points of the
fiber $\pi^{-1}(x)\subset W$, and this Galois action is isomorphic to
the left-translation action on $G/H$. The Frobenius $\frob_{x,p}$ at
$x$, seen as an element of $G$, also permutes the points of the fiber,
and the subset of rational points $\pi^{-1}(x)\cap W(\Fp)$ correspond
bijectively to the fixed points of $\frob_{x,p}$, and hence the number
of fixed points of $\frob_{x,p}$ acting on $G/H$ is equal to the
number of conjugates of $\frob_{x,p}$ that are in $H$.
\par
More generally, if we consider the function
$$
\theta\,:\,
\begin{cases}
G\lra \bar{\Qq}_{\ell}\\
g\mapsto \begin{cases}
1&\text{ if $g$ is conjugate to \emph{some} $h\in H$}\\
0&\text{ otherwise,}
\end{cases}
\end{cases}
$$
the same argument implies that we have
$$
K(x)=\theta(\frob_{x,p})
$$
for all $x\in U(\Fp)$. 
\par
The function $\theta$ is invariant under $G$-conjugation. Hence, by
character theory (since $\ell\not=p$, the $\bar{\Qq}_{\ell}$-linear
representations of $G$ can be identified with the $\Cc$-linear
representations) there exist coefficients $c_{\rho}$ such that
$$
\theta=\sum_{\rho}{c_{\rho}\chi_{\rho}}
$$
where $\rho$ runs over isomorphism classes of irreducible
$\bar{\Qq}_{\ell}$-linear representations 
$$
\rho\,:\, G\lra \GL(V_{\rho})
$$
of $G$ and $\chi_{\rho}=\Tr\rho$ denotes the character of $\rho$. By
composition
$$
\Lambda_{\rho}\,:\, \pi_1(U)\lra \pi_1(U)/\pi_1(W)\simeq
G\fleche{\rho} \GL(V_{\rho})
$$
each $\rho$ determines an $\ell$-adic lisse sheaf $\Lambda_{\rho}$ on
$U$ which is pointwise pure of weight $0$ and satisfies
$$
\chi_{\rho}(\frob_{x,p})=\frtr{\Lambda_{\rho}}{x}{\Fp}
$$
for all $x\in U(\Fp)$. We therefore obtain
$$
K(x)=\sum_{\rho}{c_{\rho}K_{\rho}}
$$
for $x\in U(\Fp)$, where $K_{\rho}$ is the trace function of
$\Lambda_{\rho}$.  
\par
We rearrange this slightly for convenience. Let $\mathcal{T}$ denote
the set of $\rho$ such that $\Lambda_{\rho}$ is geometrically
trivial. We know that $K_{\rho}$ is a constant of weight $0$, say
$\alpha_{\rho}$, for $\rho\in\mathcal{T}$, and we define
$\sheaf{F}_1=\bar{\Qq}_{\ell}$, so $K_1(x)=1$, and
$$
c_1=\sum_{\rho\in\mathcal{T}}{c_{\rho}\alpha_{\rho}}.
$$
\par
Then we enumerate arbitrarily 
$$
\{\rho\notin\mathcal{T}\}=\{\rho_2,\ldots, \rho_k\}
$$ 
and take $\sheaf{F}_i=j_*\Lambda_{\rho_i}$ where $j\,:\, U\injecte
\Aa^1$ is the open immersion, and $c_i=c_{\rho_i}$.  This gives the
desired decomposition~(\ref{eq-decomp-charfun}) with
$S=\tilde{S}(\Fp)$, which has $\leq |\tilde{S}|\leq \deg P$ elements.
\par
We now bound the numerical invariants in this decomposition. First,
note that the number of non-zero summands is at most the number of
$\rho$, i.e, the number of conjugacy classes in $G$, and hence is
bounded in terms of $\deg P$ only.  For any $\rho$ we have
$$
|c_{\rho}|=\Bigl|\frac{1}{|G|}
\sum_{g\in G}{\theta(g)\chi_{\rho}(g)}
\Bigr|\leq \dim\rho\leq \sqrt{|G|}
$$
which is bounded in terms of $\deg P$ only (using very trivial bounds
$|\chi_{\rho}(g)|\leq \dim \rho$, $|\theta(g)|\leq 1$ and the fact
that the sum of squares of $\dim\rho$ is equal to $|G|$).  And since
$p\nmid |G|$, all sheaves $\Lambda_{\rho}$ are tame, and since they
are unramified outside $S$, we get
$$
\cond(\Lambda_{\rho})\leq |S|+\dim\rho
$$
which is again bounded in terms of $\deg P$ only.
\par
Moreover, none of the sheaves $\Lambda_{\rho}$ can
contain a Jordan-H\"older factor geometrically isomorphic to
$\sheaf{L}_{\chi(X)}\otimes\sheaf{L}_{\psi(X)}$ with $\psi$
non-trivial, since the $\Lambda_{\rho}$ are tamely ramified
everywhere. If we assume that $P$ is squarefree,
$0$ is not a critical value, and all the sheaves $\Lambda_{\rho}$
are unramified at $0$ and therefore cannot have a non-trivial Kummer
sheaf as (geometric) Jordan-Hölder factor. Thus the sheaf $\sheaf{F}_i$ does
not contain an exceptional factor in this case.
\par
We conclude by proving~(\ref{eq-size-c1}): we have
\begin{align*}
  |\mathcal{P}\cap U(\Fp)|&=\sum_{x\in U(\Fp)}\theta(\frob_{x,p})\\
  &=\sum_{\rho}c_{\rho}\sum_{x\in U(\Fp)}{K_{\rho}(x)}\\
  &=c_1|U(\Fp)|+\sum_{\rho\notin \mathcal{T}} {c_{\rho}\sum_{x\in
      U(\Fp)}{K_{\rho}(x)}}.
\end{align*}
\par
For each $\rho$ which is not geometrically trivial, we can apply the
Riemann Hypothesis to the inner sum, which shows it is $\ll p^{1/2}$
with an implicit constant that depends only on $\deg P$ (since the
conductor of $\Lambda_{\rho}$ is bounded in terms of $\deg P$
only). Since the number of $\rho$ and the constants $c_{\rho}$ are
also bounded in terms of $\deg P$ only, we obtain
$$
c_1=\frac{|\mathcal{P}\cap
  U(\Fp)|}{|U(\Fp)|}+O(p^{1/2}|U(\Fp)|^{-1}),
$$
hence the result since $p-\deg P\leq |U(\Fp)|\leq p$. 
\end{proof}

\end{document}